\documentclass[12pt,a4paper]{article}

\usepackage[english]{babel}
\usepackage{amsmath}
\usepackage{amssymb}
\usepackage{amsthm}
\usepackage{mathdots}
\usepackage{mathabx}
\usepackage{tikz}
\usepackage{tikz-cd}
\usepackage{enumitem}
\usepackage[whole]{bxcjkjatype}
\usepackage{csquotes}
\usepackage{changepage}
\usepackage{color}
\definecolor{linkcolor}{rgb}{0,0,0.6}
\usepackage[colorlinks=true,
pdfstartview=FitV,
linkcolor= linkcolor,
citecolor= linkcolor,
urlcolor= linkcolor,
hyperindex=true,
hyperfigures=false]
{hyperref}

\def\MYTITLE{Cohomology of the BT strata in the unramified unitary RZ space of signature $(1,n-1)$}
\title{Cohomology of the Bruhat-Tits strata in the unramified unitary Rapoport-Zink space of signature $(1,n-1)$}
\author{J.Muller}
\date{October 1, 2021} 

\usepackage{fancyhdr} 

\makeatletter
\renewcommand\tableofcontents{%
  \null\hfill\textbf{\Large\contentsname}\hfill\null\par
  \@mkboth{\MakeUppercase\contentsname}{\MakeUppercase\contentsname}%
  \@starttoc{toc}%
}
\makeatother

\usepackage[
backend=bibtex,
style=alphabetic,
]{biblatex}
\usepackage{chngcntr}
\counterwithin*{paragraph}{section}

\newcommand{\zarabic}[1]{\ifnum\value{#1}=0 \else.\arabic{#1}\fi}

\usepackage[explicit]{titlesec}
\titlespacing{\paragraph}{0pt}{10pt}{.1em}[]

\usepackage{ytableau}
\usepackage{scalerel}
\def\msquare{\mathord{\scalerel*{\Box}{gX}}}

\begin{document}

\fontsize{12pt}{17pt}\selectfont

\setlength{\parindent}{2em} 
\setlength{\parskip}{0.5em}

\newtheorem*{theo}{Theorem}
\newtheorem*{prop}{Proposition}
\newtheorem*{lem}{Lemma}
\newtheorem*{corol}{Corollary}
\newtheorem*{conj}{Conjecture}

\theoremstyle{remark}
\newtheorem*{rk}{Remark}
\newtheorem*{rks}{Remarks}

\theoremstyle{definition}
\newtheorem*{defi}{Definition}
\newtheorem*{notation}{Notation}

\newcommand{\lcal}{\mathcal{L}}


\maketitle

\begin{center}

\parbox{15cm}{\small
\textbf{Abstract} : \it In [Inventiones mathematicae, 184 (2011)], Vollaard and Wedhorn defined a stratification on the special fiber of the unitary unramified PEL Rapoport-Zink space with signature $(1,n-1)$. They constructed an isomorphism between the closure of a stratum, called a closed Bruhat-Tits stratum, and a Deligne-Lusztig variety which is not of classical type. In this paper, we describe the $\ell$-adic cohomology groups over $\overline{\mathbb Q_{\ell}}$ of these Deligne-Lusztig varieties, where $\ell \not = p$. The computations involve the spectral sequence associated to the Ekedahl-Oort stratification of a closed Bruhat-Tits stratum, which translates into a stratification by Coxeter varieties whose cohomology is known. Eventually, we find out that the irreducible representations of the finite unitary group which appear inside the cohomology contribute to only two different unipotent Harish-Chandra series, one of them belonging to the principal series.}

\vspace{0.5cm}
\end{center}

\tableofcontents

\newpage

\noindent Rapoport-Zink spaces are geometric objects which can be seen as deformation spaces for a $p$-divisible group equipped with additional structures. They are formal schemes over the ring of integers of a $p$-adic field, and they are constructed by means of a moduli problem which grants them with commuting actions from some $p$-adic and Galois groups. Therefore, the étale cohomology of these spaces carries representations of these groups simultaneously, and it is expected to realize a local version of Langlands correspondance. Computing this cohomology is an arduous problem in general. So far it has only been entirely described in a few special cases such as the Lubin-Tate tower or the Drinfeld space ; in particular both of them correspond to Rapoport-Zink spaces of EL type.\\
\noindent The difficulty in studying the cohomology of the Rapoport-Zink spaces is maybe reflected by the lack of precise understanding of their geometry in general. However, for some specific choices of the set of data, the resulting moduli space may display some nice geometric features, giving hopes that their cohomology could be accessible. It is the case of the unitary unramified PEL Rapoport-Zink space $\mathcal M$ of signature $(1,n-1)$, whose special fiber $\mathcal M_{\mathrm{red}}$ is described by the Bruhat-Tits stratification constructed by Vollaard and Wedhorn in a series of two papers \cite{vw1} and \cite{vw2}. This stratification $\{\mathcal M_{\Lambda}\}_{\Lambda}$ has two interesting features. On the one hand, the closed strata $\mathcal M_{\Lambda}$ are indexed by the set of vertices $\Lambda$ of the Bruhat-Tits building of a $p$-adic group of unitary similitudes $J$ defined by the PEL data. On the other hand, each individual closed stratum $\mathcal M_{\Lambda}$ is isomorphic to a Deligne-Lusztig variety. They usually show up in Deligne-Lusztig theory, whose aim is the classification of the irreducible representations of finite groups of Lie type. In particular, the cohomology of these varieties has been extensively studied in the past decades. In a series of two papers, we aim at exploiting these geometric observations in order to link the cohomology theories of Deligne-Lusztig varieties and of Rapoport-Zink spaces.\\
Our strategy in order to examine the cohomology of the Rapoport-Zink space $\mathcal M$ takes place in two steps: first we compute the cohomology groups $\mathrm{H}^{\bullet}_{c}(\mathcal M_{\Lambda},\overline{\mathbb Q_{\ell}})$ (with $\ell \not = p$) of the closed strata; second we use the combinatorics of the Bruhat-Tits building to get information on the cohomology of $\mathcal M$. More precisely, in the second stage we introduce the analytical generic fiber $\mathcal M^{\mathrm{an}}$. It is covered with the analytical tubes $U_{\Lambda}$ of the closed strata $\mathcal M_{\Lambda}$. These are open subdomains of $\mathcal M^{\mathrm{an}}$ whose cohomology coincides with the cohomology of the closed strata up to a suitable Tate twist and shift in degrees. Through the \v{C}ech spectral sequence associated to the open cover $\{U_{\Lambda}\}_{\Lambda}$, we prove the semisimplicity of the Frobenius action on $\mathrm H^{\bullet}_c(\mathcal M^{\mathrm{an}},\overline{\mathbb Q_{\ell}})$, and we determine the cuspidal supports of its irreducible subquotients as a smooth representation of the group of unitary similitudes $J$. It also turns out that these cohomology groups are not $J$-admissible in general.\\
Lastly, the $p$-adic uniformization theorem relates the Rapoport-Zink space $\mathcal M^{\mathrm{an}}$ with the basic stratum of an associated PEL unitary Shimura variety through a geometric isomorphism. It induces a Hochschild-Serre type spectral sequence on the cohomology, through which we compute the individual cohomology groups of the basic stratum in the case $n=3$ and $4$. In particular, we find out that some automorphic representations occur with a multiplicity depending on $p$ which is a completely new phenomenon.\\
In the present paper, we carry out the first step of the strategy described above, namely we compute the cohomology of the individual closed Bruhat-Tits strata by exploiting Deligne-Lusztig theory. The second step and the results stated above can be found in the sequel \cite{muller}.\\

\noindent Let $q$ be a power of the prime number $p$. Let $\mathbf G$ be a connected reductive group over an algebraic closure $\mathbb F$ of $\mathbb F_{q}$. Assume that $\mathbf G$ is equipped with an $\mathbb F_q$-structure induced by a Frobenius morphism $F: \mathbf G \rightarrow \mathbf G$. For $\mathbf P$ a parabolic subgroup of $\mathbf G$, the associated generalized parabolic Deligne-Lusztig variety is defined by 
$$X_{\mathbf P} := \{g\mathbf P \in \mathbf G/\mathbf P \,|\, g^{-1}F(g) \in \mathbf P F(\mathbf P)\}.$$
Usually, parabolic Deligne-Lusztig varieties have been studied with the additional assumption that $\mathbf P$ contains a Levi complement $\mathbf L$ such that $F(\mathbf L) = \mathbf L$. Indeed, they are used to define the Deligne-Lusztig induction and restriction functors between the categories of representations of $\mathbf L^F$ and of $\mathbf G^F$, see for instance \cite{dm}. However, the closed Bruhat-Tits strata constructed by Vollaard and Wedhorn are isomorphic to Deligne-Lusztig varieties $X_{\mathbf P}$ associated to parabolic subgroups $\mathbf P$ which do not satisfy this assumption. We call them \enquote{generalized} and to our knowledge, their cohomology has not been studied so far.\\
The closed Bruhat-Tits strata $\mathcal M_{\Lambda}$ are isomorphic to generalized Deligne-Lusztig varieties $X_I(\mathrm{id})$ associated to finite unitary groups $\mathrm{U}_{2d+1}(p)$ in an odd number of variables (see \ref{defDLvariety} for the notations). Although only the case $q = p$ is relevant in the context of Vollaard and Wedhorn's paper  \cite{vw2}, we will work in this paper with a general $q$. In loc. cit. the authors defined yet another stratification on each individual stratum. It is called the Ekedahl-Oort stratification and it gives a decomposition $X_I(\mathrm{id}) \simeq \bigsqcup_{0\leq t\leq d} X_{I_t}(w_t)$ into locally closed subvarieties. It turns out that each Ekedahl-Oort stratum $X_{I_t}(w_t)$ is isomorphic to a Deligne-Lusztig variety which is not generalized. Moreover, they are closely related to Coxeter varieties whose cohomology is known thanks to the work of Lusztig in \cite{cox}. The Ekedahl-Oort stratification on $X_I(\mathrm{id})$ induces a spectral sequence on the cohomology, through which we are able to entirely compute the individual cohomology groups in terms of representations of $\mathrm U_{2d+1}(q)$. The representations which occur are all unipotent, and these are classified by partitions of $2d+1$ or equivalently by Young diagrams, see \cite{ls} and \cite{fong}. Given a partition $\lambda = (\lambda_1 \geq \ldots \geq \lambda_r)$ of $2d+1$ with $\lambda_r > 0$, the associated irreducible unipotent representation of $\mathrm{U}_{2d+1}(q)$ is denoted by $\rho_{\lambda}$. We may now state our main result, whose proof covers the section 5 of the paper. In the statement, the prime number $\ell$ is different from $p$, the field $\mathbb F$ is an algebraic closure of $\mathbb F_{q}$ and $\mathrm{Frob}$ is the geometric Frobenius relative to $\mathbb F_{q^2}$ acting on the cohomology groups.

\noindent \begin{theo} 
The following statements hold. 
\begin{enumerate}[label=\upshape (\arabic*), topsep = 0pt]
		\item The cohomology group $\mathrm H_c^i(X_I(\mathrm{id}) \otimes \mathbb F,\overline{\mathbb Q_{\ell}})$ is zero unless $0 \leq i \leq 2d$. There is an isomorphism $\mathrm H_c^i(X_I(\mathrm{id})\otimes \mathbb F,\overline{\mathbb Q_{\ell}}) \simeq \mathrm H_c^{2d-i}(X_I(\mathrm{id}) \otimes \mathbb F,\overline{\mathbb Q_{\ell}})^{\vee}(d)$ which is equivariant for the actions of $\mathrm{Frob}$ and of $\mathrm U_{2d+1}(q)$. 
		\item The Frobenius element $\mathrm{Frob}$ acts like multiplication by $(-q)^i$ on $\mathrm H_c^i(X_I(\mathrm{id}) \otimes \mathbb F,\overline{\mathbb Q_{\ell}})$. 
		\item For $0\leq i \leq d$ we have 
		$$\mathrm H_c^{2i}(X_I(\mathrm{id}) \otimes \mathbb F,\overline{\mathbb Q_{\ell}}) = \bigoplus_{s=0}^{\min(i,d - i)} \rho_{(2d + 1 - 2s, 2s)}.$$
		\item For $0\leq i \leq d - 1$ we have 
		$$\mathrm H_c^{2i+1}(X_I(\mathrm{id}) \otimes \mathbb F,\overline{\mathbb Q_{\ell}}) = \bigoplus_{s=0}^{\min(i,d - 1 - i)} \rho_{(2d - 2s, 2s + 1)}.$$
	\end{enumerate}
\end{theo}

\noindent In particular, when the index is even all the representations in the cohomology groups contribute to the unipotent principal series, but when the index is odd the representations belong to the unipotent series determined by a minimal Levi complement of $\mathrm{U}_{2d+1}(q)$ which is not a torus.\\

\noindent Throughout the paper, we fix $q$ a power of an odd prime number $p$. If $k$ is a perfect field extension of $\mathbb F_q$, we denote by $\sigma: x \mapsto x^q$ the $q$-th power Frobenius of $\mathrm{Gal}(k/\mathbb F_q)$. We fix an algebraic closure $\mathbb F$ of $\mathbb F_q$.\\

\noindent \textbf{\textsc{Acknowledgement:}} This paper is part of a PhD thesis under the supervision of Pascal Boyer (Université Sorbonne Paris Nord) and Naoki Imai (University of Tokyo). I am grateful for their wise guidance throughout the research. I also wish to adress special thanks to Olivier Dudas (Université de Paris) who gave me precious support regarding Deligne-Lusztig theory. He taught me the subtleties of this field, and guided me through the vast literature with precise references in order to carry out the computations of the cohomology groups.

\section{The generalized Deligne-Lusztig variety $X_I(\mathrm{id})$}

\paragraph{}\label{notations}Let $\mathbf G$ be a connected reductive group over $\mathbb F$. Let $F$ be a Frobenius morphism defining an $\mathbb F_q$-structure on it. If $\mathbf H$ is an $F$-stable subgroup of $\mathbf G$, we denote by $H := \mathbf H^F \simeq \mathbf H(\mathbb F_q)$ its group of $\mathbb F_q$-rational points. We fix a pair $(\mathbf T,\mathbf B)$ consisting of a maximal torus $\mathbf T$ contained in a Borel subgroup $\mathbf B$, both of them being $F$-stable. Such a pair always exists up to $G=\mathbf G^F$-conjugation. We obtain a Coxeter system $(\mathbf W,\mathbf S)$ on which $F$ acts, where $\mathbf W = \mathbf W(\mathbf T)$ is the Weyl group attached to $\mathbf T$ and $\mathbf S$ is the set of simple reflexions. It can be identified with the Weyl group of $\mathbf G$ as defined in \cite{dl}. Let $\ell$ denote the length function on $\mathbf W$ relative to $\mathbf S$. For $I\subset \mathbf S$, we write $\mathbf P_I, \mathbf U_I, \mathbf L_I$ respectively for the standard parabolic subgroup of type $I$, for its unipotent radical and for its unique Levi complement containing $\mathbf T$. We also write $\mathbf W_I$ for the parabolic subgroup of $\mathbf W$ generated by the simple reflexions in $I$. Recall that an element $w\in \mathbf W$ is said to be $I$-reduced (resp. reduced-$I$) if for every $v\in \mathbf W_I$, we have $\ell(vw) = \ell(v) + \ell(w)$ (resp. $\ell(wv) = \ell(w) + \ell(v)$). The set of $I$-reduced (resp. reduced-$I$) elements is denoted by ${}^I\mathbf W$ (resp. $\mathbf W^{I}$). If $I,I'\subset \mathbf S$, an element is said to be $I$-reduced-$I'$ if it belongs to ${}^I\mathbf W^{I'} := {}^I\mathbf W \cap \mathbf W^{I'}$.

\paragraph{}\label{defDLvariety}We recall the definition of Deligne-Lusztig varieties from \cite{bonnafe}. If $\mathbf P$ is any parabolic subgroup of $\mathbf G$, the associated generalized parabolic Deligne-Lusztig variety is
$$X_{\mathbf P} := \{g\mathbf P\in \mathbf G/\mathbf P\,|\,g^{-1}F(g)\in \mathbf P F(\mathbf P)\}.$$
When these varieties were first introduced in \cite{dl} only the case of Borel subgroups was considered, hence the adjective \enquote{parabolic}. Moreover, parabolic Deligne-Lusztig varieties have mostly been studied with the additional assumption that $\mathbf P$ contains an $F$-stable Levi complement, see for instance \cite{dm}. This is not required by the definition above, hence the adjective \enquote{generalized}.\\
Using the Coxeter system as above, one may give an equivalent description of these varieties. For $I,I'\subset \mathbf S$ the generalized Bruhat decomposition is an isomorphism
$$\mathbf P_I\backslash \mathbf G / \mathbf P_{I'} = \bigsqcup_{w\in{}^I\mathbf W^{I'}} \mathbf P_I\backslash \mathbf P_Iw\mathbf P_{I'}/\mathbf P_{I'} \simeq \mathbf W_{I}\backslash \mathbf W / \mathbf W_{I'}.$$
For $w\in {}^I\mathbf W^{F(I)}$, the generalized parabolic Deligne-Lusztig varieties is defined by
$$X_I(w) = \{g\mathbf P_I\in \mathbf G/\mathbf P_I \,|\, g^{-1}F(g)\in \mathbf P_IwF(\mathbf P_I)\}.$$
The families of varieties $X_{\mathbf P}$ and $X_I(w)$ are the same and \cite{bonnafe} explains how to go from one description to the other. The case $I = \emptyset$ corresponds to usual Deligne-Lusztig varieties in $\mathbf G/\mathbf B$. Moreover, the additional assumption regarding the existence of a rational Levi complement translates into the equation 
\begin{equation}\label{CompatibilityCondition}
w^{-1}Iw = F(I), \tag{$*$}
\end{equation}
which is a compatibility condition between the parameters $w$ and $I$. The variety $X_I(w)$ is defined over $\mathbb F_{q^{\iota}}$, where $\iota$ is the least integer such that $F^{\iota}(I) = I$ and $F^{\iota}(w) = w$. 

\paragraph{}In this paragraph, we compute the dimension of a generalized Deligne-Lusztig variety $X_I(w)$. For any $w\in \mathbf W$, let $\ell(w)$ denote the length of $w$ with respect to $\mathbf S$.

\begin{prop}
For $I\subset \mathbf S$ and $w\in {}^I\mathbf W^{F(I)}$, we have 
$$\dim X_I(w) = \ell(w) + \dim \mathbf G/\mathbf P_{I\cap wF(I)w^{-1}} - \dim \mathbf G/\mathbf P_{I}.$$
\end{prop}

\noindent Let us introduce a few more notations. If $I,I' \subset \mathbf S$, the generalized Bruhat decomposition implies that the $\mathbf G$-orbits for the diagonal action on $\mathbf G/\mathbf P_I \times \mathbf G/\mathbf P_{I'}$ are given by 
$$\mathcal O_{I,I'}(w) := \{(g\mathbf P_I,h\mathbf P_{I'}) \,|\, g^{-1}h \in \mathbf P_Iw\mathbf P_{I'}\}$$
for $w\in {}^I\mathbf W^{I'}$. The Deligne-Lusztig variety $X_I(w)$ can be seen as the intersection of $\mathcal O_{I,F(I)}(w)$ with the graph of the Frobenius $F:\mathbf G/\mathbf P_I \rightarrow \mathbf G/\mathbf P_{F(I)}$. This intersection is transverse, see \cite{dl} 9.11 (in loc. cit. the proof deals with the case $I = \emptyset$, but it generalizes to any $I$). Thus, the proposition follows from the following lemma and the fact that $\dim \mathbf P_I = \dim \mathbf P_{F(I)}$. 

\begin{lem}
For $I,I' \subset \mathbf S$ and $w\in {}^I\mathbf W^{I'}$, we have 
$$\dim \mathcal O_{I,I'}(w) = \ell(w) + \dim \mathbf G/\mathbf P_{I\cap wI'w^{-1}}.$$
\end{lem}

\begin{proof}
Recall that for $I \subset \mathbf S$, the standard parabolic subgroup of type $I$ decomposes as a union of Bruhat cells $\mathbf P_I = \mathbf B\mathbf W_I\mathbf B$, and any Bruhat cell $\mathbf Bw\mathbf B$ has dimension $\dim \mathbf B + \ell(w)$. Therefore $\dim \mathbf P_I = \dim \mathbf B + \ell(I)$ where $\ell(I)$ denotes the maximal length of elements of $\mathbf W_I$.\\
Let $I,I'$ and $w$ be as in the lemma. Consider the first projection $\mathcal O_{I,I'}(w) \rightarrow \mathbf G/\mathbf P_I$ which is a surjective morphism with fibers isomorphic to $\mathbf P_Iw\mathbf P_{I'}/\mathbf P_{I'}$. It is flat since $\mathbf G \rightarrow \mathbf G/\mathbf P_I$ is faithfully flat, and the pullback $\mathcal O_{I,I'}(w) \times_{\mathbf G/\mathbf P_I} \mathbf G$ is isomorphic to $\mathbf G \times \mathbf P_Iw\mathbf P_{I'}/\mathbf P_{I'}$. We have 
$$\mathbf P_Iw\mathbf P_{I'} = \mathbf B\mathbf W_I\mathbf Bw\mathbf B\mathbf W_{I'}\mathbf B = \mathbf B \mathbf W_Iw\mathbf W_{I'}\mathbf B,$$
therefore the dimension of a fiber is given by 
$$\dim \mathbf P_Iw\mathbf P_{I'}/\mathbf P_{I'} = \dim \mathbf P_Iw\mathbf P_{I'} - \dim \mathbf P_{I'} = \max_{v\in \mathbf W_Iw\mathbf W_{I'}} \ell(v) - \ell(I').$$
Since $w$ is $I$-reduced-$I'$, according to \cite{dmbook} Lemma 3.2.2, any element $v \in W_Iw\mathbf W_{I'}$ can uniquely be written as $v = xwy$ such that $x\in \mathbf W_I, y \in \mathbf W_{I'}$ and $xw$ is reduced-$I'$. In particular $\ell(v) = \ell(x) + \ell(w) + \ell(y)$. It follows that 
$$\max_{v\in \mathbf W_Iw\mathbf W_{I'}} \ell(v) = \ell(w) + \max_{x \in \mathbf W_I \cap \mathbf W^{I'}w^{-1}} \ell(x) + \ell(I').$$
We prove that $\mathbf W_I \cap \mathbf W^{I'}w^{-1} = \mathbf W_I \cap \mathbf W^{I\cap wI'w^{-1}}$.\\

\noindent Let $x \in \mathbf W_I \cap \mathbf W^{I'}w^{-1}$, we show that $x$ is reduced-$I\cap wI'w^{-1}$. Let $s \in I\cap wI'w^{-1}$, so that we can write $s = wtw^{-1}$ for some $t\in I'$. Then $xsw = xwt$. Since $xs \in \mathbf W_I$ and $w$ is $I$-reduced, the left hand side has length $\ell(xs) + \ell(w)$. On the other hand, since $t\in I'$ and $xw$ is reduced-$I'$, the right hand side has length $\ell(xw) + 1 = \ell(x) + \ell(w) + 1$. Therefore $\ell(xs) = \ell(x) + 1$ as expected. \\
For the other inclusion, let $y \in \mathbf W_I \cap \mathbf W^{I\cap wI'w^{-1}}$. We show that $yw$ is reduced-$I'$. Towards a contradiction, assume that $\ell(ywt) < \ell(yw)$ for some $t\in I'$. Let $y = s_1\ldots s_r$ and $w = u_1\ldots u_{r'}$ be reduced expressions respectively of $y$ and of $w$, with the $s_i$ in $I$ and the $u_j$ in $\mathbf S$. Since $w$ is $I$-reduced, the concatenation of both reduced expressions give a reduced expression of $yw$. By the exchange condition (see \cite{dmbook} 2.1.2), we have 
$$ywt = s_1 \ldots \widehat{s_i}\ldots s_r w \text { or } yu_1\ldots \widehat{u_j}\ldots u_{r'}$$
for some $1\leq i \leq r$ or $1\leq j \leq r'$, where $\widehat{\cdot}$ denotes the product with one omitted term. The second case is impossible, since after simplifying $y$ it would contradict the fact that $w$ is reduced-$I'$. \\
Let us write $s :=  y^{-1}s_1 \ldots \widehat{s_i}\ldots s_r \in \mathbf W_I$, so that we have 
$$wt =  sw.$$
The left hand side has length $\ell(w) + 1$, and the right hand side has length $\ell(s) + \ell(w)$. It follows that $s\in I$ has length $1$. Therefore $s = wtw^{-1} \in I\cap wI'w^{-1}$. Eventually, we have $\ell(ys) = \ell(y) + 1$ since $y$ is reduced-$(I\cap wI'w^{-1})$. This is absurd, because $ys = s_1 \ldots \widehat{s_i}\ldots s_r$ has length $r-1 = \ell(y)-1$. \\

\noindent To conclude the proof, we recall the following general fact. If $(\mathbf W,\mathbf S)$ is a Coxeter system and $K \subset \mathbf S$, then the product map $\mathbf W^K \times \mathbf W_K \xrightarrow{\sim} \mathbf W$ mapping $(w^K,w_K)$ to $w^Kw_K$ is a bijection. In particular we have 
$$\max_{w\in \mathbf W} \ell(w) = \max_{w^K \in \mathbf W^K} \ell(w^K) + \max_{w_K \in \mathbf W_K} \ell(w_K).$$
We apply this to the Coxeter system $(\mathbf W_I,I)$ and $K = I\cap wI'w^{-1}$. It follows that 
$$\max_{x \in \mathbf W_I \cap \mathbf W^{I'}w^{-1}} \ell(x) = \max_{x \in \mathbf W_I \cap \mathbf W^{I\cap wI'w^{-1}}} \ell(x) =  \ell(I) - \ell(I\cap wI'w^{-1}).$$
Putting things together, have proved that 
\begin{align*}
\dim \mathcal O_{I,I'}(w) & = \dim \mathbf G/\mathbf P_I + \dim \mathbf P_Iw\mathbf P_{I'}/\mathbf P_{I'} \\
& = \dim \mathbf G - \dim \mathbf B - \ell(I) + \max_{v\in \mathbf W_Iw\mathbf W_{I'}} \ell(v) - \ell(I') \\
& = \dim \mathbf G - \dim \mathbf B - \ell(I) + \ell(w) + \max_{x \in \mathbf W_I \cap \mathbf W^{I'}w^{-1}} \ell(x) \\
& = \dim \mathbf G - \dim \mathbf B - \ell(I\cap wI'w^{-1}) + \ell(w)\\
& = \dim \mathbf G/\mathbf P_{I\cap wI'w^{-1}} + \ell(w).
\end{align*}
\end{proof}

\begin{rk}
In \cite{vw2} 4.4, the formula given by the authors for the dimension of $\mathcal O_{I,I'}(w)$, and as a consequence for the Deligne-Lusztig variety $X_I(w)$ as well, contained a mistake.
\end{rk}

\paragraph{}\label{finitehermitianspace} Let $d$ be a nonnegative integer and let $V$ be a $(2d+1)$-dimensional $\mathbb F_{q^2}$-vector space. Let $(\cdot,\cdot):V\times V \to \mathbb F_{q^2}$ be a non-degenerate hermitian form on $V$. This hermitian structure on $V$ is unique up to isomorphism. In particular, we may once and for all a basis $\mathcal B$ of $V$ in which $(\cdot,\cdot)$ is described by the square matrix $\dot{w}_0$ of size $2d +1$, having $1$ on the anti-diagonal and $0$ everywhere else. If $k$ is a perfect field extension of $\mathbb F_{q^2}$, we may extend the pairing to $V_k := V\otimes_{\mathbb F_{q^2}} k$ by setting 
$$(v\otimes x,w\otimes y):=xy^{\sigma}(v,w)\in k$$
for all $v,w\in V$ and $x,y\in k$. If $U$ is a subspace of $V_k$ we denote by $U^{\perp}$ its orthogonal, that is the subspace of all vectors $x\in V_k$ such that $(x,U)=0$.\\
Let $J$ denote the finite group of Lie type $\mathrm{U}(V,(\cdot,\cdot))$. It is defined as the group of $F$-fixed points of $\mathbf J := \mathrm{GL}(V)_{\mathbb F}$ with $F$ a non-split Frobenius morphism. Using the basis $\mathcal B$, the group $\mathbf J$ is identified with $\mathrm{GL}_{2d+1}$ with $\mathbb F_q$-structure induced by the Frobenius morphism $F(M):=\dot{w}_0(M^{(q)})^{-t}\dot{w}_0$. Here, $M^{(q)}$ denotes the matrix $M$ having all coefficients raised to the power $q$. We may then identify $J$ with the usual finite unitary group $\mathrm{U}_{2d+1}(q)$.\\
The pair $(\mathbf T,\mathbf B)$ consisting of the maximal torus of diagonal matrices and the Borel subgroup of upper-triangular matrices is $F$-stable. The Weyl system of $(\mathbf T,\mathbf B)$ may be identified with $(\mathfrak S_{2d+1},\mathbf S)$ in the usual manner, where $\mathbf S$ is the set of simple transpositions $s_i:=(i\quad i+1)$ for $1\leq i \leq 2d$. Under this identification, the Frobenius acts on $\mathbf W$ as the conjugation by the element $w_0$, characterized for having the maximal length. It satisfies $w_0(i) = 2d + 2 -i$, and a natural representative of $w_0$ in the normalizer of $\mathbf T$ is no other than $\dot{w}_0$. Since $w_0$ has order $2$, the action of the Frobenius on $\mathbf W$ is involutive. It also preserves the simple reflexions with the formula $F(s_i) = s_{2d+1-i}$.

\paragraph{}\label{isomorphismDLvariety} We define the following subset of $\mathbf S$ 
$$I:=\{s_1,\ldots ,s_{d} ,s_{d+2},\ldots ,s_{2d}\} = \mathbf S\setminus \{s_{d+1}\}.$$
We have $F(I) = \mathbf S \setminus \{s_{d}\} \not = I$. We consider the generalized Deligne-Lusztig variety $X_{I}(\mathrm{id})$. It corresponds to the variety denoted $Y_{\Lambda}$ in \cite{vw2} 4.5. It has dimension $d$ and it does not satisfy the compatibility condition \eqref{CompatibilityCondition}.

\begin{prop}[\cite{vw2} 4.4]
The variety $X_{I}(\mathrm{id})$ is defined over $\mathbb F_{q^2}$ and it is projective, smooth, geometrically irreducible of dimension $d$.
\end{prop}

\noindent Although the proposition in loc. cit. is only stated in the case $q = p$, the arguments carry over to general $q$. The geometric irreducibility is a consequence of the criterion proved in \cite{bonnafe}.

\begin{rk}
Even though the dimension formula for generalized Deligne-Lusztig varieties in \cite{vw2} is wrong, it does give the correct result in the case of $X_I(\mathrm{id})$. It is because for $w = \mathrm{id}$, we have $I\cap wF(I)w^{-1} = I \cap F(I)$. Therefore, that mistake does not change anything regarding the validity of the authors' work.\\
For example, we may consider the Deligne-Lusztig variety $X_I(s_2s_1)$ for $\mathrm{U}_{3}(\mathbb F_q)$ with $I = \{s_1\}$. It is classical so that $\dim X_I(s_2s_1) = \ell(s_2s_1) = 2$. However, we have $\mathbf P_{I\cap F(I)} = \mathbf B$ and $\dim \mathbf G/\mathbf B = 3$ whereas $\dim \mathbf G/\mathbf P_I = 2$, so that the formula of loc. cit. says that $X_I(s_2s_1)$ would be of dimension $2+3-2 = 3$. 
\end{rk}

\paragraph{}Rational points of Deligne-Lusztig varieties associated to a unitary group $\mathrm{U}$ over $\mathbb F_q$ can be described in terms of vectorial flags, in a certain relative position with respect to their image by the Frobenius. Let $k$ be a perfect field extension of $\mathbb F_{q^2}$. According to \cite{vw1} 2.12, the Frobenius acts on a flag $\mathcal F$ in $V_k$ by sending it to its orthogonal flag $\mathcal F^{\perp}$. Explicitely, we have 
$$
\begin{tikzcd}[column sep=scriptsize]
\mathcal F\;  &[-25pt] : \quad \{0\}  &[-25pt] \subset &[-25pt] \mathcal F_1 &[-25pt] \subset &[-25pt] \ldots &[-25pt] \subset &[-25pt] \mathcal F_r &[-25pt] \subset &[-25pt] V_k,\\[-15pt]
\mathcal F^{\perp}\; &[-25pt] : \quad \{0\}  &[-25pt] \subset &[-25pt] \mathcal F^{\perp}_r &[-25pt] \subset &[-25pt] \ldots &[-25pt] \subset &[-25pt] \mathcal F^{\perp}_1 &[-25pt] \subset &[-25pt] V_k.
\end{tikzcd}
$$
Here, given our choice of $I$, a $k$-rational point of $X_I(\mathrm{id})$ corresponds to a flag of the type $$\mathcal F: \{0\} \subset U \subset V_k$$ with $U$ having dimension $d+1$, and which is of relative position $\mathrm{id}$ with respect to $\mathcal F^{\perp}$. This precisely means that $U$ must contain $U^{\perp}$.

\begin{prop}
The $k$-rational points of $X_I(\mathrm{id})$ are given by 
$$X_I(\mathrm{id})(k) \simeq \{U\subset V_k \,|\, \dim U = d+1 \text{ and }U^{\perp}\subset U\}.$$
\end{prop}

\paragraph{}\label{EOstratification} In \cite{vw2} 5.3, the authors defined the \textbf{Ekedahl-Oort stratification} on the Deligne-Lusztig variety $X_I(\mathrm{id})$. By loc. cit. Corollary 5.12, it turns out that each stratum is itself isomorphic to a parabolic Deligne-Lusztig variety which is not generalized. They are defined as follows.\\
For $0 \leq t \leq d$, we define the subset 
$$I_{t}:=\{s_1,\ldots ,s_{d - t -1},s_{d + t + 2},\ldots ,s_{2d}\} \subset \mathbf S.$$
The subset $I_{t}$ consists of all $2d$ simple reflexions in $\mathbf S$, except that we removed the $2t+2$ ones in the middle. Thus, it has cardinality $2(d - t - 1)$. In particular, it is empty for $t = d$ or $d-1$. We also define the cycle $w_{t} := (d + t + 1 \quad d + t \, \ldots \, d +1)$. Its decomposition into simple reflexions is $w_{t} = s_{d+1}\ldots s_{d + t}$. When $t = 0$, it is the identity. We note that even though $I_{d} = I_{d-1} = \emptyset$, we still have $w_{d} \not = w_{d-1}$.\\
One may check that $F(I_{t}) = I_{t}$ and that $w_{t}$ belongs to $^{I_{t}}\mathbf W^{I_{t}}$. Moreover, the compatibility condition \eqref{CompatibilityCondition} is satisfied for the pair $(I_{t},w_{t})$. Indeed, the reduced decomposition for $w_{t}$ does not use any simple reflexion that is adjacent to those in $I_{t}$.

\noindent \begin{prop}[\cite{vw2} 3.3 and 5.3]
The Deligne-Lusztig variety $X_{I_t}(w_t)$ is defined over $\mathbb F_{q^2}$ and has dimension $t$. There is a natural immersion $X_{I_t}(w_t) \hookrightarrow X_I(\mathrm{id})$ inducing a stratification 
$$X_I(\mathrm{id}) = \bigsqcup_{0 \leq t \leq d} X_{I_t}(w_t).$$
The closure of the stratum $X_{I_t}(w_t)$ is the union of all the strata $X_{I_s}(w_s)$ for $s \leq t$.
\end{prop}

\paragraph{}Following the proof of Theorem 2.15 of \cite{vw1}, we can describe the stratification at the level of rational points. Let $k$ be a perfect field extension of $\mathbb F_{q^2}$. Because of the choice of $I_{t}$, a $k$-point of $X_{I_{t}}(w_{t})$ is a flag 
$$
\begin{tikzcd}[column sep=scriptsize]
\mathcal F\; : \quad \{0\}  \subset \mathcal F_{-t-1} \subset \ldots \subset \mathcal F_{-1} \subset \mathcal F_1 \subset \ldots \subset \mathcal F_{t+1} \subset V_k
\end{tikzcd}
$$
with $\dim(\mathcal F_{-i}) = d+1-i$ and $\dim(\mathcal F_i) = d+i$ for $1\leq i \leq t +1$, and which is in relative position $w_t$ with respect to $\mathcal F^{\perp}$. It means that we have a diagram of the following type. 

$$
\begin{tikzcd}[column sep=scriptsize]
    \mathcal F : &[-25pt] \mathcal F_{-t-1} \arrow[equal]{d} &[-25pt] \subset \ldots \subset &[-25pt] \mathcal F_{-1} \arrow[equal]{d} &[-25pt] \subset &[-25pt] \mathcal F_1 \arrow[d,equal,"/"{anchor=center,sloped}] \arrow[hook]{drr} &[-25pt] \subset &[-25pt] \mathcal F_2 \arrow[d,equal,"/"{anchor=center,sloped}] \arrow[hook]{dr} &[-25pt] \subset  \ldots \arrow[hook]{dr} \subset &[-25pt] \mathcal F_{t} \arrow[d,equal,"/"{anchor=center,sloped}] \arrow[hook]{drr} &[-25pt] \subset &[-25pt] \mathcal F_{t+1} \arrow[equal]{d} \\[10pt]
    \mathcal F^{\perp} : &[-25pt] \mathcal F_{t+1}^{\perp} &[-25pt] \subset \ldots \subset &[-25pt] \mathcal F_{1}^{\perp} &[-25pt] \subset &[-25pt] \tau(\mathcal F_1) &[-25pt] \subset &[-25pt] \tau(\mathcal F_2) &[-25pt] \subset  \ldots \subset &[-25pt] \tau(\mathcal F_{t}) &[-25pt] \subset &[-25pt] \tau(\mathcal F_{t+1})
\end{tikzcd}
$$

\noindent Here, $\tau := \sigma^2\cdot\mathrm{id}$ is an $\mathbb F_{q^2}$-linear automorphism of $V_k$, and it satisfies $\tau(U) = (U^{\perp})^{\perp}$ for every subspace $U \subset (V_{\Lambda})_k$. This diagram implies that $\tau(\mathcal F_i) = \mathcal F_{i-1} + \tau(\mathcal F_{i-1})$ for all $2\leq i \leq t+1$. This rewrites as $\mathcal F_{i} =  \mathcal F_{i-1} + \tau^{-1}(\mathcal F_{i-1})$. We deduce that 
$$\mathcal F_i = \sum_{l=0}^{i-1} \tau^{-l}(\mathcal F_1)$$
for all $1\leq i \leq t+1$. Thus, the whole flag is determined by the subspace $\mathcal F_1$, which has dimension $d+1$ and contains its orthogonal. The immersion $X_{I_t}(w_t) \hookrightarrow X_I(\mathrm{id})$ maps the flag $\mathcal F$ to $\mathcal F_1$.\\
Conversely, a $k$-point of $X_I(\mathrm{id})$ is given by a subspace $U \subset V_k$ of dimension $d+1$ containing its orthogonal. For $i \geq 1$ we define
$$\mathcal F_i := \sum_{l=0}^{i-1} \tau^{-1}(U) \subset V_k.$$
Then $(\mathcal F_i)_{i\geq 1}$ is a nondecreasing sequence of subspaces of $V_k$. Let $t$ be the smallest integer such that $\mathcal F_{t+1} = \mathcal F_{t+2}$. It follows that $0 \leq t \leq d$ and that $t$ is also the smallest integer such that $\mathcal F_{t+1} = \tau(\mathcal F_{t+1})$. Moreover the orthogonal $U^{\perp}$ has dimension $d$ and we have $U^{\perp} \subset U$, so that $U^{\perp} \subset (U^{\perp})^{\perp} = \tau(U)$. In particular, if $t>0$ then $U \cap \tau(U) = U^{\perp}$. Thus, we have $\dim(\mathcal F_2) = d + 2$. Similarly, we have $\dim(\mathcal F_i) = d+i$ for all $1 \leq i \leq t+1$. By setting $\mathcal F_{-i} := \mathcal F_i^{\perp}$, we obtain a flag $\mathcal F$ that is the $k$-rational point of $X_{I_t}(w_t)$ associated to $U$.

\paragraph{}\label{isomorphismEOstratum} The Deligne-Lusztig varieties $X_{I_{t}}(w_{t})$ are related to Coxeter varieties for smaller unitary groups as we now explain. We define 
$$K_{t} := \{s_1,\ldots s_{d-t-1}, s_{d-t+1},\ldots , s_{d+t}, s_{d+t+2}, \ldots , s_{2d}\} = \mathbf S \setminus \{s_{d-t}, s_{d+t+1}\}.$$
The set $K_{t}$ is obtained from $I_{t}$ by adding the $2t$ simple reflexions in the middle. It has cardinality $2d - 2$ and satisfies $F(K_{t}) = K_{t}$. We have $I_{t} \subset K_{t}$ with equality if and only if $t=0$.

\noindent \begin{prop}
There is a $\mathrm U_{2d+1}(q)$-equivariant isomorphism
$$X_{I_{t}}(w_{t}) \simeq \mathrm U_{2d+1}(q)/U_{K_{t}} \times_{L_{K_{t}}} X_{I_{t}}^{\mathbf L_{K_{t}}}(w_{t}),$$
where $X_{I_{t}}^{\mathbf L_{K_{t}}}(w_{t})$ is a Deligne-Lusztig variety for $\mathbf L_{K_{t}}$. The zero-dimensional variety $\mathrm U_{2d+1}(q)/U_{K_{t}}$ has a left action of $\mathrm U_{2d+1}(q)$ and a right action of $L_{K_{t}}$.
\end{prop}

\begin{proof}
This is an application of \cite{dm} Proposition 7.19 which is the geometric identity behind the transitivity of the Deligne-Lusztig functors. It applies to the varieties $X_{I_t}(w_t)$ because they satisfy the compatibility condition \eqref{CompatibilityCondition}, and satisfies the following conditions: $K_{t}$ contains $I_{t}$, it is stable by the Frobenius and $w_{t}$ belongs to the parabolic subgroup $\mathbf W_{K_{t}} \simeq \mathfrak S_{d-t}\times \mathfrak S_{2t+1} \times \mathfrak S_{d-t} \subset \mathfrak S_{2d +1}$.
\end{proof}

\paragraph{}\label{decompositionDLvarieties} The Levi complement $\mathbf L_{K_{t}}$ is isomorphic to the product $\mathrm{GL}_{d-t} \times \mathrm{GL}_{2t +1} \times \mathrm{GL}_{d - t}$ as a reductive group over $\mathbb F$. Given a matrix $M = \mathrm{diag}(A,C,B) \in \mathbf L_{K_{t}}$, we have $F(M) = \mathrm{diag}(F(B),F(C),F(A))$, where we still denote by $F$ the Frobenius morphism for smaller linear groups. Writing $\mathbf H$ for the product of the two $\mathrm{GL}_{d-t}$ factors, we have $\mathbf L_{K_{t}} \simeq \mathbf H \times \mathrm{GL}_{2t +1}$ and both factors inherit an $\mathbb F_q$-structure by means of $F$. We have $L_{K_{t}} \simeq \mathrm{GL}_{d-t}(q^2) \times \mathrm U_{2t +1}(q)$, the first factor corresponding to $H$. \\
The Weyl group of $\mathbf L_{K_{t}}$ is isomorphic to $\mathbf W_{\mathbf H}\times \mathfrak S_{2t +1}$ where $\mathbf W_{\mathbf H} \simeq \mathfrak S_{d-t}\times \mathfrak S_{d-t}$ is the Weyl group of $\mathbf H$. Via this decomposition, the permutation $w_{t}$ corresponds to $\mathrm{id}\times \widetilde{w_{t}}$, where $\widetilde{w_{t}}$ is the restriction of $w_{t}$ to $\{d - t + 1, \ldots , d + t + 1\}$. Similarly, the set of simple reflexions $\mathbf S$ decomposes as $\mathbf S_{\mathbf H}\sqcup \widetilde{\mathbf S}$, the second term corresponding to the simple reflexions in $\mathfrak S_{2t+1}$. Then, we have $I_{t} = \mathbf S_{\mathbf H} \sqcup \emptyset$.\\
The Deligne-Lusztig variety for $\mathbf L_{K_{t}}$ decompose accordingly as the following product 
$$X_{I_{t}}^{\mathbf L_{K_{t}}}(w_{t}) = X_{\mathbf S_{\mathbf H}}^{\mathbf H}(\mathrm{id}) \times X_{\emptyset}^{\mathrm U_{2t+1}(q)}(\widetilde{w_{t}}).$$
The variety $X_{\mathbf S_{\mathbf H}}^{\mathbf H}(\mathrm{id})$ is just a point, whereas $X_{\emptyset}^{\mathrm U_{2t+1}(q)}(\widetilde{w_{t}})$ is a Deligne-Lusztig variety for the unitary group of size $2t +1$. We observe that the permutation $\widetilde{w_{t}}$ is a \textbf{Coxeter element} in $\mathfrak S_{2t+1}$, ie. the product of exactly one simple reflexion for each orbit of the Frobenius. Deligne-Lusztig varieties attached to Coxeter elements are called \textbf{Coxeter varieties}, and their cohomology with coefficients in $\overline{\mathbb Q_{\ell}}$ where $\ell$ is a prime number different from $p$ are well understood thanks to the work of Lusztig in \cite{cox}. Before stating the results of loc. cit. we recall parts of the representation theory of finite unitary groups.

\section{Irreducible unipotent representations of the finite unitary group}

\paragraph{}\label{classification} In this section, we recall the classification of the irreducible unipotent representations of the finite unitary group and we explain the underlying combinatorics.\\
We use the notations from \ref{notations}. For $w\in \mathbf W$, let $\dot{w}$ be a representative of $w$ in the normalizer $\mathrm N_{\mathbf G}(\mathbf T)$ of $\mathbf T$. By the Lang-Steinberg theorem, one can find $g\in \mathbf G$ such that $\dot{w} = g^{-1}F(g)$. Then $^g\mathbf T := g\mathbf T g^{-1}$ is another $F$-stable maximal torus, and $w \in \mathbf W$ is said to be the \textbf{type} of $^g \mathbf T$ with respect to $\mathbf T$. Every $F$-stable maximal torus arises in this manner. According to \cite{dl} Corollary 1.14, the $G$-conjugacy class of $^g \mathbf T$ only depends on the $F$-conjugacy class of the image $w$ of the element $g^{-1}F(g) \in \mathrm N_{\mathbf G}(\mathbf T)$ in the Weyl group $\mathbf W$. Here, two elements $w$ and $w'$ in $\mathbf W$ are said to be $F$-conjugates if there exists some element $u \in \mathbf W$ such that $w = u w' F(u)^{-1}$.\\
For every $w\in \mathbf W$, we fix $\mathbf T_w$ an $F$-stable maximal torus of type $w$ with respect to $\mathbf T$. The Deligne-Lusztig induction of the trivial representation of $\mathbf T_w$ is the virtual representation of $G$ defined by the formula 
$$R_w := \sum_{i\geq 0} (-1)^i\mathrm H^i_c(X_{\emptyset}(w))$$
where $X_{\emptyset}(w)$ is a Deligne-Lusztig variety for $\mathbf G$ as defined in \ref{defDLvariety}. According to \cite{dl} Theorem 1.6, the virtual representation $R_w$ only depends on the $F$-conjugacy class of $w$ in $\mathbf W$. An irreducible representation of $G$ is said to be \textbf{unipotent} if it occurs in $R_w$ for some $w\in \mathbf W$. The set of isomorphism classes of unipotent representations of $G$ is usually denoted $\mathcal E(G,1)$ following Lusztig's notations.

\paragraph{}Assume that the Coxeter graph of the reductive group $\mathbf G$ is a union of subgraphs of type $A_m$ (for various $m$). Let $\widecheck{\mathbf W}$ be the set of isomorphism classes of irreducible representations of its Weyl group $\mathbf W$. The action of the Frobenius $F$ on $\mathbf W$ induces an action on $\widecheck{\mathbf W}$, and we consider the fixed point set $\widecheck{\mathbf W}^F$. Then, the following classification theorem is well known.

\noindent \begin{theo}[\cite{ls} Theorem 2.2]
There is a bijection between $\widecheck{\mathbf W}^F$ and the set of isomorphism classes of irreducible unipotent representations of $G = \mathbf G^F$. 
\end{theo}

\noindent We recall how the bijection is constructed. If $V\in \widecheck{\mathbf W}^F$ is an irreducible $F$-stable representation of $\mathbf W$, according to loc. cit. there is a unique automorphism $\widetilde{F}$ of $V$ of finite order such that 
$$R(V) := \frac{1}{|\mathbf W|}\sum_{w\in \mathbf W} \mathrm{Trace}(w\circ \widetilde{F} \,|\, V)R_w$$
is an irreducible representation of $G$. Then the map $V \mapsto R(V)$ is the desired bijection.\\
In the case $\mathbf G = \mathrm{GL}_n$ with the Frobenius morphism $F$ being either standard or twisted (ie. $G = \mathrm{GL}_n(q)$ or $\mathrm{U}_n(q)$), we have an equality $\widecheck{\mathbf W}^F = \widecheck{\mathbf W}$. Moreover, the automorphism $\widetilde{F}$ is the identity in the former case and multiplication by $w_0$ on the latter, where $w_0$ is the element of maximal length in $\mathbf W$. Thus, in both cases the irreducible unipotent representations of $G$ are classified by the irreducible representations of the Weyl group $\mathbf W\simeq \mathfrak S_n$, which in turn are classified by partitions of $n$ or equivalently by Young diagrams. We now recall the underlying combinatorics behind the representation theory of the symmetric group. A general reference is \cite{james}.

\paragraph{}\label{partition}A partition of $n$ is a tuple $\lambda = (\lambda_1 \geq \ldots \geq \lambda_r)$ with $r\geq 1$ and the $\lambda_i$'s are positive integers such that $\lambda_1 + \ldots + \lambda_r = n$. The integer $n$ is called the length of the partition and it is also denoted by $|\lambda|$. If a partition has a series of repeating integers, it is common to write it shortly with an exponent. For instance, the partition $(3,3,2,2,1)$ of $11$ will be denoted $(3^2,2^2,1)$. Partitions of $n$ are naturally identified with Young diagrams of size $n$. The diagram attached to $\lambda$ has $r$ rows consisting successively of $\lambda_1, \ldots ,\lambda_r$ boxes.\\
To any partition $\lambda$ of $n$, one can naturally associate an irreducible representation $\chi_{\lambda}$ of the symmetric group $\mathfrak S_n$. An explicit construction is given, for instance, by the notion of Specht modules as explained in \cite{james} 7.1. In particular, the character $\chi_{(n)}$ is trivial while the character $\chi_{(1^n)}$ is the signature.

\paragraph{}\label{MurnaghanNakayama} We recall the Murnaghan-Nakayama rule which gives a recursive formula to evaluate the characters $\chi_{\lambda}$. We first need to introduce skew Young diagrams. Consider a pair $\lambda$ and $\mu$ of two partitions respectively of integers $n+k$ and $k$. Assume that the Young diagram of $\mu$ is contained in the Young diagram of $\lambda$. By removing the boxes corresponding to $\mu$ from the diagram of $\lambda$, one finds a shape consisting of $n$ boxes denoted by $\lambda\setminus \mu$. Any such shape is called a \textbf{skew Young diagram} of size $n$. It is said to be connected if one can go from a given box to any other by moving in a succession of adjacent boxes.\\
For example, consider the partition $\lambda = (3^2,2^2,1)$ and let us define the partitions $\mu_1 = (2^2)$, $\mu_2 = (3,1^2)$ and $\mu_3 = (2,1)$. The diagrams below correspond, from left to right, to the skew Young diagrams $\lambda \setminus \mu_i$ for $i=1,2,3$.
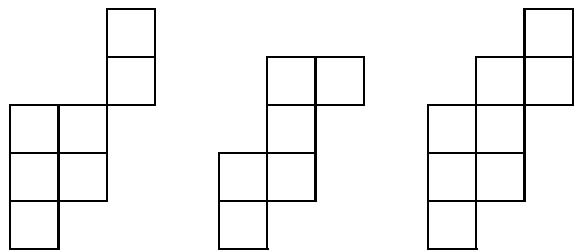
\begin{center}\ydiagram{2+1,2+1,2,2,1} \qquad \ydiagram{0,1+2,1+1,2,1} \qquad \ydiagram{2+1,1+2,2,2,1}\end{center}
The skew Young diagram $\lambda\setminus \mu_1$ is not connected, whereas the others are connected. A skew Young diagram is said to be a \textbf{border strip} if it is connected and if it does not contain any $2\times 2$ square. The \textbf{height} of a border strip is defined as its number of rows minus $1$. For instance, among the three skew Young diagrams above only $\lambda\setminus \mu_2$ is a border strip. Its size is $6$ and its height is $3$. \\
The characters $\chi_{\lambda}$ are class functions, so we only need to specify their values on conjugacy classes of the symmetric group $\mathfrak S_n$. These conjugacy classes are also naturally labelled by partitions of $n$. Indeed, up to ordering any permutation $\sigma \in \mathfrak S_n$ can be uniquely decomposed as a product of $r\geq 1$ cycles $c_1,\ldots ,c_r$ with disjoint supports. We denote by $\nu_i$ the cycle length of $c_i$ and we order them so that $\nu_1 \geq \ldots \geq \nu_r$. We allow cycles to have length $1$, so that the union of the supports of all the $c_i$'s is $\{1,\ldots ,n\}$. Thus, we obtain a partition $\nu = (\nu_1,\ldots,\nu_r)$ of $n$ which is called the \textbf{cycle type} of the permutation $\sigma$. Two permutations are conjugates in $\mathfrak S_n$ if and only if they share the same cycle type. We denote by $\chi_{\lambda}(\nu)$ the value of the character $\chi_{\lambda}$ on the conjugacy class labelled by $\nu$.

\noindent \begin{theo}[Murnaghan-Nakayama rule]
Let $\lambda$ and $\nu$ be two partitions of $n$. We have
$$\chi_{\lambda}(\nu) = \sum_{S} (-1)^{\mathrm{ht}(S)}\chi_{\lambda\setminus S}(\nu\setminus \nu_1),$$
where $S$ runs over the set of all border strips of size $\nu_1$ in the Young diagram of $\lambda$, such that removing $S$ from $\lambda$ gives again a Young diagram. Here, the integer $\mathrm{ht}(S) \in \mathbb Z_{\geq 0}$ is the height of the border stip $S$, the Young diagram $\lambda \setminus S$ is the one obtained by removing $S$ from $\lambda$, and $\nu\setminus\nu_1$ is the partition of $n - \nu_1$ obtained by removing $\nu_1$ from $\nu$. 
\end{theo}

\noindent Applying the Murnaghan-Nakayama rule in successions results in the value of $\chi_{\lambda}(\nu)$. We see in particular that $\chi_{(n)}$ is the trivial character whereas $\chi_{(1^n)}$ is the signature. We illustrate the computations with $\lambda = (3^2,2^2,1)$ and $\nu = (4^2,3)$. There are only two elligible border strips of size $4$ in the diagram of $\lambda$, as marked below.
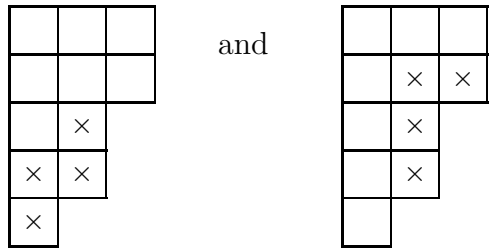
\begin{center}
\ydiagram{3,3,1}*[\times]{0,0,1+1,2,1} \qquad and \qquad \ydiagram{3,1,1,1,1}*[\times]{0,1+2,1+1,1+1}
\end{center}
Both border strips have height $2$. Thus, the formula gives
$$\chi_{(3^2,2^2,1)}(4^2,3) = \chi_{(3^2,1)}(4,3) + \chi_{(3,1^4)}(4,3).$$
In each of the two Young diagrams obtained after removal of the border strips, there is only one elligible strip of size $4$, and eventually the three last remaining boxes form the final border strip of size $3$.
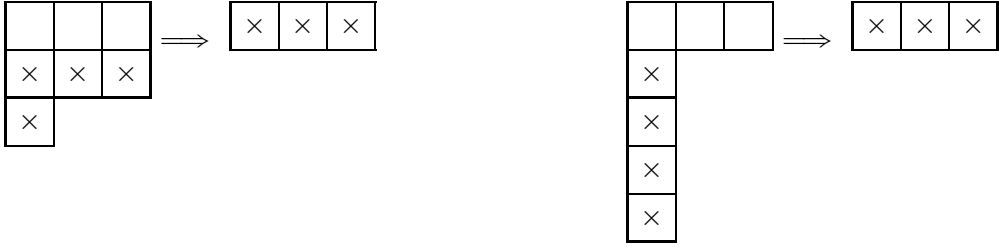
\begin{center}
\ydiagram{3}*[\times]{0,3,1} $\implies$ \ydiagram[\times]{3} \qquad \qquad \qquad \qquad \ydiagram{3}*[\times]{0,1,1,1,1} $\implies$  \ydiagram[\times]{3}
\end{center}
Taking the heights of the border strips into account, we find 
\begin{align*}
\chi_{(3^2,1)}(4,3) = - \chi_{(3)}(3) = - \chi_{\emptyset} = -1, & & \chi_{(3,1^4)}(4,3) = - \chi_{(3)}(3) = - \chi_{\emptyset} = -1.
\end{align*}
Here, $\emptyset$ denotes the empty partition. The computation finally gives $\chi_{(3^2,2^2,1)}(4^2,3) = -2$. 

\paragraph{}\label{dimensionformula} The irreducible unipotent representation of $\mathrm U_n(q)$ (resp. $\mathrm{GL}_n(q)$) associated to $\chi_{\lambda}$ by the bijection of \ref{classification} Theorem is denoted by $\rho_{\lambda}^{\mathrm U}$ (resp. $\rho_{\lambda}^{\mathrm{GL}}$). The partition $(n)$ corresponds to the trivial representation and $(1^n)$ to the Steinberg representation in both cases. We will omit the superscript when the group we are talking about is clear from the context.\\
The degrees of the representations $\rho_{\lambda}^{\mathrm{GL}}$ and $\rho_{\lambda}^{\mathrm U}$ are given by expressions known as \textbf{hook formula}. Given a box $\msquare$ in the Young diagram of $\lambda$, its \textbf{hook length} $h(\msquare)$ is $1$ plus the number of boxes lying below it or on its right. For instance, in the following figure the hook length of every box of the Young diagram of $\lambda = (3^2,2^2,1)$ has been written inside it. \\
\ytableausetup{mathmode, centertableaux}
\begin{center}
\begin{ytableau}
7 & 5 & 2 \\
6 & 4 & 1 \\
4 & 2 \\
3 & 1 \\
1
\end{ytableau}
\end{center}

\noindent \begin{prop}[\cite{geck2} Propositions 4.3.1 and 4.3.5] 
Let $\lambda = (\lambda_1 \geq \ldots \geq \lambda_r)$ be a partition of $n$. The degrees of the irreducible unipotent representations $\rho_{\lambda}^{\mathrm{GL}}$ and $\rho_{\lambda}^{\mathrm U}$, respectively of $\mathrm{GL}_n(q)$ and $\mathrm U_n(q)$, are given by the following formulas 
\begin{align*}
\deg(\rho_{\lambda}^{\mathrm{GL}}) = q^{a(\lambda)}\frac{\prod_{i=1}^n q^i - 1}{\prod_{\msquare \in \lambda} q^{h(\msquare)} - 1}, & & \deg(\rho_{\lambda}^{\mathrm{U}}) = q^{a(\lambda)}\frac{\prod_{i=1}^n q^i - (-1)^i}{\prod_{\msquare \in \lambda} q^{h(\msquare)} - (-1)^{h(\msquare)}},
\end{align*}
where $a(\lambda) = \sum_{i=1}^r (i-1)\lambda_i$. 
\end{prop}

\paragraph{}We recall from \cite{geck} 3.1 and 3.2 some definitions on classical Harish-Chandra theory. A parabolic subgroup of $G$ is a subgroup $P\subset G$ such that there exists an $F$-stable parabolic subgroup $\mathbf P$ of $\mathbf G$ with $P = \mathbf P^F$. A Levi complement of $G$ is a subgroup $L \subset G$ such that there exists an $F$-stable Levi complement $\mathbf L$ of $\mathbf G$, contained inside some $F$-stable parabolic subgroup, such that $L = \mathbf L^F$. Any parabolic subgroup $P$ of $G$ has a Levi complement $L$.\\
Let $L = \mathbf L^F$ be a Levi complement of $G$ inside a parabolic subgroup $P = \mathbf P^F$. Let $U = \mathbf U^F$ be the $F$-fixed points of the unipotent radical $\mathbf U$ of $\mathbf P$. The \textbf{Harish-Chandra induction and restriction functors} are defined by the following formulas.
\begin{align*}
\mathrm{R}^G_{L\subset P}: \mathrm{Rep}(L) & \to \mathrm{Rep}(G) & {}^*\mathrm{R}^G_{L\subset P}: \mathrm{Rep}(G) & \to \mathrm{Rep}(L) \\
 \sigma & \mapsto \mathbb C[G/U]\otimes_{\mathbb C[L]}\sigma & \rho & \mapsto \mathrm{Hom}_G(\mathbb C[G/U],\rho)
\end{align*}
\noindent Here, $\mathrm{Rep}(G)$ is the category of complex representations of $G$, and similarily for $\mathrm{Rep}(L)$. These two functors are adjoint, and up to isomorphism they do not depend on the choice of the parabolic subgroup $P$ containing the Levi complement $L$. For this reason, we will denote the functors $\mathrm R^G_L$ and ${}^*\mathrm R^G_L$ instead.\\
An irreducible representation of $G$ is called \textbf{cuspidal} if its Harish-Chandra restriction to any proper Levi complement is zero. We consider pairs $(L,X)$ where $L$ is a Levi complement of $G$ and $X$ is an irreducible representation of $L$. We define an order on the set of such pairs by setting $(L,X) \leq (M,Y)$ if $L\subset M$ and if $X$ occurs in the Harish-Chandra restriction of $Y$ to $L$. A pair is said to be \textbf{cuspidal} if it is minimal with respect to this order, in which case $X$ is a cuspidal representation of $L$. If $(L,X)$ is a cuspidal pair, we will denote by $[L,X]$ its conjugacy class under $G$.\\
Given a cuspidal pair $(L,X)$ of $G$, its associated \textbf{Harish-Chandra series} $\mathcal E(G,(L,X))$ is defined as the set of isomorphism classes of irreducible constituents in the induction of $X$ to $G$. Each series is non empty. Two of them are either disjoint or equal, the latter occuring if and only if the two cuspidal pairs are conjugates in $G$. Thus, the series are indexed by the conjugacy classes of cuspidal pairs $[L,X]$. Moreover, the isomorphism class of any irreducible representation of $G$ belongs to some Harish-Chandra series. Thus, Harish-Chandra series form a partition of the set of isomorphism classes of irreducible representations of $G$. If $\rho$ is an irreducible representation of $G$, the conjugacy class $[L,X]$ corresponding to the series to which $\rho$ belongs is called the \textbf{cuspidal support} of $\rho$. If $T$ denotes a maximal torus in $G$, then the series $\mathcal E(G,(T,1))$ is called the \textbf{unipotent principal series} of $G$.

\paragraph{}\label{HarishChandraSeries}For the general linear group $\mathrm{GL}_n(q)$, there is no unipotent cuspidal representation unless $n=1$, in which case the trivial representation is cuspidal. Moreover, the unipotent representations all belong to the principal series. The situation for the unitary group is very different. First, by \cite{classical} 9.2 and 9.4 there exists an irreducible unipotent cuspidal representation of $\mathrm U_n(q)$ if and only if $n$ is an integer of the form $n = \frac{x(x+1)}{2}$ for some $x\geq 0$, and when that is the case it is the one associated to the partition $\Delta_x := (x, x-1,\ldots,1)$, whose Young diagram has the distinctive shape of a reversed staircase. Here, as a convention $\mathrm U_0(q)$ denotes the  trivial group.\\
For example, here are the Young diagrams of $\Delta_1,\Delta_2$ and $\Delta_3$. Of course, the one of $\Delta_0$ the empty diagram.
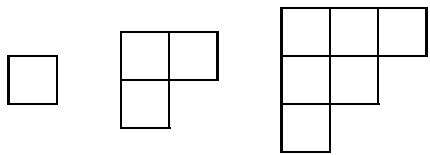
\begin{center}\ydiagram{1} \quad \quad \ydiagram{2,1} \quad \quad \ydiagram{3,2,1}\end{center}

\noindent Furthermore the unipotent representations decompose non trivially into various Harish-Chandra series, as we recall from \cite{geck} 4.3.\\
We consider an integer $x\geq 0$ such that $n$ decomposes as $n = 2a + \frac{x(x+1)}{2}$ for some $a\geq 0$. We also consider the standard Levi complement $L_x \simeq \mathrm{GL}_1(q^2)^a \times \mathrm U_{\frac{x(x+1)}{2}}(q)$ which corresponds to the choice of simple reflexions $s_{a+1}, \ldots , s_{n-a-1}$. We write $\rho_x$ for the inflation of $\rho^{\mathrm U}_{\Delta_x}$ to an irreducible representation of $L_x$. Then $\mathcal E(\mathrm U_n(q),1)$ decomposes as the disjoint union of all the Harish-Chandra series $\mathcal E(\mathrm U_n(q), (L_x,\rho_x))$ for all possible choices of $x$. With these notations, the principal unipotent series corresponds to $x=0$ if $n$ is even and to $x=1$ if $n$ is odd.

\paragraph{}Given an irreducible unipotent representation $\rho_{\lambda}$ of $\mathrm U_n(q)$, there is a combinatorical way of determining the Harish-Chandra series to which it belongs. We consider the Young diagram of $\lambda$. We call \textbf{domino} any pair of adjacent boxes in the diagram. It may be either vertical or horizontal. We remove dominoes from the rim of the diagram of $\lambda$ so that the resulting shape is again a Young diagram, until one can not proceed further. This process results in the Young diagram of the partition $\Delta_x$ for some $x\geq 0$, and it is called the \textbf{$2$-core} of $\lambda$. It does not depend on the successive choices for the dominoes. Then, the representation $\rho_{\lambda}$ belongs to the series $\mathcal E(\mathrm U_n(q),(L_x,\rho_x))$ if and only if $\lambda$ has $2$-core $\Delta_x$. \\
For instance, the diagram $\lambda = (3^2,2^2,1)$ has $2$-core $\Delta_1$, as it can be determined by the following steps. We put crosses inside the successive dominoes that we remove from the diagram. Thus, the unipotent representation $\rho_{\lambda}$ of $\mathrm{U}_{11}(q)$ belongs to the unipotent principal series $\mathcal E(\mathrm U_{11}(q),(L_1,\rho_1))$.
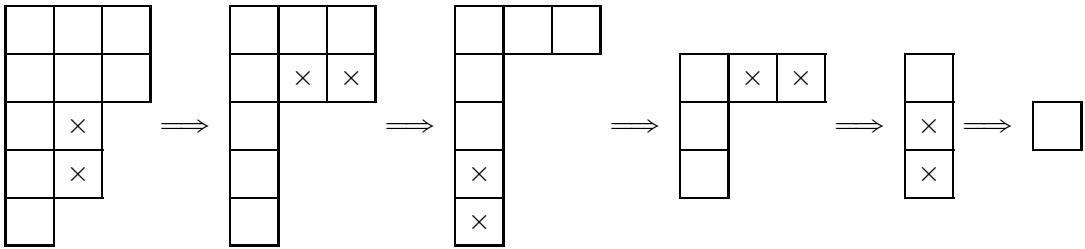
\begin{center}
\ydiagram{3,3,1,1,1}*[\times]{0,0,1+1,1+1} $\implies$ \ydiagram{3,1,1,1,1}*[\times]{0,1+2} $\implies$ \ydiagram{3,1,1}*[\times]{0,0,0,1,1} $\implies$ \ydiagram{1,1,1}*[\times]{1+2} $\implies$ \ydiagram{1}*[\times]{0,1,1} $\implies$ \ydiagram{1}
\end{center}

\section{Computing Harish-Chandra induction of unipotent representations in the finite unitary group}

\paragraph{}\label{comparisontheorem} In this paragraph, we recover the notations from \ref{notations}. We recall from \cite{geck} 3.2 how Harish-Chandra induction of unipotent representations can be explicitely computed. Let $W = \mathbf W^F$ be the Weyl group of $G$. It is still a Coxeter group, whose set of simple reflexions $S$ is identified with the set of $F$-orbits on $\mathbf S$. Let $(L,X)$ be a cuspidal pair of $G$. The \textbf{relative Weyl group of $L$} is given by $W_G(L) := N_{\mathbf G}(\mathbf L)^F/L\subset W$. The relative Weyl group of the pair $(L,X)$, also called \textbf{the ramification group of $X$} in \cite{howlett}, is the subgroup $W_G(L,X)$ of $W_G(L)$ consisting of elements $w$ such that $wX \simeq X$, where $wX$ denotes the representation $wX(g) := X(wgw^{-1})$ of $L$. It is yet again a Coxeter group if $\mathbf G$ has a connected center or if $X$ is unipotent. \\
Theorem 3.2.5 of \cite{geck} establishes an isomorphism between the endomorphism algebra of the induced representation $\mathrm R_{L}^{G}(X)$ and the complex group ring of the ramification group $W_G(L,X)$. In particular, this gives an bijection between the Harish-Chandra series $\mathcal E(G,(L,X))$ and the set $\mathrm{Irr}(W_G(L,X))$ of isomorphism classes of irreducible complex characters of $W_G(L,X)$. These bijections for $G$ and for various Levi complements in $G$ can be chosen to be compatible with Harish-Chandra induction. This is known as Howlett and Lehrer's comparison theorem which was proved in \cite{howlett}.

\noindent \begin{theo}[\cite{geck} Comparison Theorem 3.2.7] 
Let $(L,X)$ be a cuspidal pair for the finite group of Lie type $G$. For every Levi complement $M$ in $G$ containing $L$, the bijection between $\mathrm{Irr}(W_M(L,X))$ and $\mathcal E(M,(L,X))$ can be taken so that the diagrams 
$$
\begin{tikzcd}
    \mathbb Z\mathcal E(G,(L,X)) \arrow{rr}{\sim} & & \mathbb Z\mathrm{Irr}(W_{G}(L,X)) \\[10pt]
    \mathbb Z\mathcal E(M,(L,X)) \arrow{rr}{\sim} \arrow{u}{\mathrm R_M^{G}} & & \mathbb Z\mathrm{Irr}(W_{M}(L,X)) \arrow[swap]{u}{\mathrm{Ind}}
\end{tikzcd}
\quad\quad
\begin{tikzcd}
    \mathbb Z\mathcal E(G,(L,X)) \arrow{rr}{\sim} \arrow[swap]{d}{^*\mathrm R_M^{G}} & & \mathbb Z\mathrm{Irr}(W_{G}(L,X)) \arrow{d}{\mathrm{Res}} \\[10pt]
    \mathbb Z\mathcal E(M,(L,X)) \arrow{rr}{\sim}  & & \mathbb Z\mathrm{Irr}(W_{M}(L,X)) 
\end{tikzcd}
$$
are commutative. Here, $\mathrm{Ind}$ and $\mathrm{Res}$ on the right-hand side of the diagrams are the classical induction and restriction functors for representations of finite groups.
\end{theo}

\noindent In other words, computing Harish-Chandra induction and restrictions of representations in $G$ can be entirely done at the level of the associated Coxeter groups. In order to use this statement for unitary groups, we need to make the horizontal arrows explicit and to understand the combinatorics behind induction and restriction of the irreducible representations of the relevant Coxeter groups. This has been explained consistently in \cite{fong} for classical groups.

\paragraph{}\label{unitaryrelativeWeylgroup} We focus on the case of the unitary group. Let $x\geq 0$ such that $n = 2a + \frac{x(x+1)}{2}$ for some $a\geq 0$. We consider the cuspidal pair $(L_x,\rho_x)$ as in \ref{HarishChandraSeries}, with $L_x = \mathrm{GL}_1(q^2)^a\times \mathrm{U}_{\frac{x(x+1)}{2}}(q)$. The relative Weyl group $W_{\mathrm{U}_n(q)}(L_x)$ is isomorphic to the Coxeter group of type $B_a$, which is usually denoted by $W_a$. Indeed, the Weyl group $W_{\mathrm U_n(q)}(L_x)$ admits a presentation by elements $\sigma_1,\ldots, \sigma_{a-1}$ and $\theta$ of order $2$ satisfying the relations
\begin{align*}
\theta\sigma_1\theta\sigma_1  & = \sigma_1 \theta \sigma_1 \theta, & \theta\sigma_i & = \sigma_i \theta, & & \forall \; 2\leq i \leq m-1. \\
\sigma_i \sigma_{i+1} \sigma_i & = \sigma_{i+1} \sigma_i \sigma_{i+1}, & \sigma_i \sigma_j & = \sigma_j \sigma_i, & & \forall\; |i-j| \geq 2.
\end{align*} 
Explicitely, the element $\sigma_i$ is represented by the permutation matrix of the double transposition $(i\quad i+1)(n-i\quad n-i+1)$ and the element $\theta$ by the matrix of the transposition $(1 \quad n)$, all of which belong to $N_{\mathrm U_n(q)}(L_x)$. This presentation coincide with the Coxeter group $W_a$ of type $B_a$, see in \cite{geck2} 1.4.1. Moreover, the ramification group $W_{\mathrm U_n(q)}(L_x,\rho_x)$ is equal to the whole of $W_{\mathrm U_n(q)}(L_x) \simeq W_a$. The identification between the ramification group and the Coxeter group $W_a$ is naturally induced by the isomorphism between the absolute Weyl group $\mathbf W$ and the symmetric group $\mathfrak S_n$. In order to proceed further, we need to explain the representation theory of the group $W_a$.

\paragraph{}Let $W_a$ be a Coxeter group of type $B_a$ given with a presentation by elements $\sigma_1,\ldots ,\sigma_{a-1}$ and $\theta$ satisfying equations as in \ref{unitaryrelativeWeylgroup}. For $1\leq i \leq a-1$, we define $\theta_i = \sigma_i\ldots \sigma_1\theta\sigma_{1}\ldots \sigma_{i}$. In particular $\theta_0 = \theta$. Following \cite{geck2} 3.4.2, we define \textbf{signed blocks} to be elements of the following form. Given $k\geq 0$ and $e\geq 1$ such that $k + e \leq a$, the positive (resp. negative) block of length $e$ starting at $k$ is 
\begin{align*}
b^+_{k,e} := \sigma_{k+1}\sigma_{k+2}\ldots \sigma_{k+e-1}, & & b^-_{k,e} := \theta_k\sigma_{k+1}\sigma_{k+2}\ldots \sigma_{k+e-1}.
\end{align*}
A \textbf{bipartition} of $a$ is an ordered pair $(\alpha,\beta)$ where $\alpha$ is a partition of some integer $0\leq j \leq a$ and $\beta$ is a partition of $a - j$. Given a bipartition $(\alpha,\beta)$ of $a$ and writing $\alpha = (\alpha_1,\ldots ,\alpha_r)$ and $\beta = (\beta_1,\ldots ,\beta_s)$, we define the element 
$$w_{\alpha,\beta} := b^-_{k_1,\beta_1}\ldots b^-_{k_s,\beta_s}b^+_{k_{s+1},\alpha_1}\ldots b^+_{k_{s+r},\alpha_r}$$
where $k_1 = 0$, $k_{i+1} = k_i + \beta_i$ if $1\leq i \leq s$ and $k_{i+1} = k_i + \alpha_{i - s}$ if $s+1\leq i \leq s+r - 1$. In particular, we have $k_{r+s} + \alpha_r = a$. According to \cite{geck2} Proposition 3.4.7, the conjugacy classes in $W_a$ are labelled by bipartitions of $a$, and a representative of minimal length of the conjugacy class corresponding to the bipartition $(\alpha,\beta)$ is given by $w_{\alpha,\beta}$. Thus, the irreducible representations of $W_a$ can be labelled by bipartitions of $a$ as well. An explicit construction of these irreducible representations is given in \cite{geck2} 5.5. We will not recall it, however we may again give a method to compute the character values, similar to the Murnaghan-Nakayama formula. The character of the irreducible representation of $W_a$ associated in loc. cit. to the bipartition $(\alpha,\beta)$ of $a$ will be denoted $\chi_{\alpha,\beta}$. If $(\gamma,\delta)$ is another bipartition of $a$, we denote by $\chi_{\alpha,\beta}(\gamma,\delta)$ the value of the character $\chi_{\alpha,\beta}$ on the conjugacy class of $W_a$ labelled by $(\gamma,\delta)$.\\
One can think of a bipartition $(\alpha,\beta)$ of $a$ as an ordered pair of two Young diagrams of combined size $a$. A \textbf{border strip} of a bipartition $(\alpha,\beta)$ is a border strip either of the partition $\alpha$ or of $\beta$. The height of a border strip is defined in the same way.

\noindent \begin{theo}[\cite{geck2} Theorem 10.3.1]
Let $(\alpha,\beta)$ and $(\gamma,\delta)$ be two bipartitions of $a$. If $\gamma \not = \emptyset$, let $\epsilon = 1$ and let $x$ be the last integer in the partition $\gamma$. If $\gamma = \emptyset$, let $\epsilon = -1$ and let $x$ be the last integer of the partition $\delta$. We have 
$$\chi_{\alpha,\beta}(\gamma,\delta) = \sum_S (-1)^{\mathrm{ht}(S)}\epsilon^{f_S}\chi_{(\alpha,\beta)\setminus S}((\gamma,\delta)\setminus x),$$
where $S$ runs over the set of all border strips of size $x$ in the bipartition $(\alpha,\beta)$, such that removing $S$ from $(\alpha,\beta)$ gives again a pair of Young diagrams. Here, the pair of Young diagrams $(\alpha,\beta)\setminus S$ is the one obtained after removing $S$, and $(\gamma,\delta)\setminus x$ is the bipartition obtained by removing $x$ from $(\gamma,\delta)$. Eventually, the integer $f_S$ is $0$ if $S$ is a border strip of $\alpha$, and it is $1$ if $S$ is a border strip of $\beta$.
\end{theo}

\noindent Applying this formula in successions results in the value of $\chi_{(\alpha,\beta)}(\gamma,\delta)$. In particular, one sees that $\chi_{(a),\emptyset}$ is the trivial character and $\chi_{\emptyset,(1^a)}$ is the signature character of $W_a$. We illustrate the computations with $(\alpha,\beta) = ((3,1^2),(4,2))$ and $(\gamma,d) = ((4),(5,2))$. There is only elligible border strip of size $4$ in the pair of diagrams $(\alpha,\beta)$, as marked below.
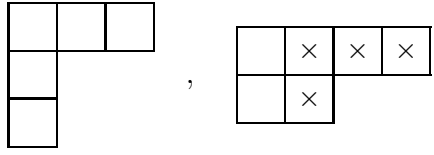
\begin{center}
\ydiagram{3,1,1}\quad , \quad \ydiagram{1,1}*[\times]{1+3,1+1}
\end{center}
This border strip $S$ has height $1$. It was taken in the diagram of $\beta$ so $f_S = 1$. Since $\gamma \not = \emptyset$ we have $\epsilon = 1$. Applying the formula, we obtain 
$$\chi_{(3,1^2),(4,2)}((4),(5,2)) = - \chi_{(3,1^2),(1^2)}(\emptyset,(5,2)).$$
We are now looking for border strips of size $2$ in the pair of diagrams of the bipartition $(3,1^2),(1^2)$. Three of them are eligible, as marked below. 
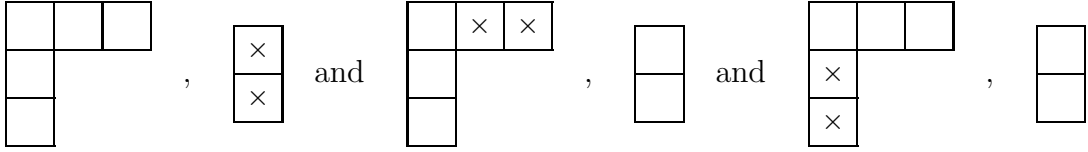
\begin{center}
\ydiagram{3,1,1}\quad , \quad \ydiagram[\times]{1,1} \quad \text{and} \quad \ydiagram{1,1,1}*[\times]{1+2} \quad , \quad \ydiagram{1,1} \quad \text{and} \quad \ydiagram{3}*[\times]{0,1,1} \quad , \quad \ydiagram{1,1}
\end{center}
These three border strips have respective heights $1,0$ and $1$. The corresponding values of $f_S$ are respectively $1$, $0$ and $0$. Moreover, the partition $\gamma$ is now empty so $\epsilon = -1$. The formula gives 
$$\chi_{(3,1^2),(1^2)}(\emptyset,(5,2)) = \chi_{(3,1^2),\emptyset}(\emptyset,(5)) + \chi_{(1^3),(1^2)}(\emptyset,(5)) - \chi_{(3),(1^2)}(\emptyset,(5)).$$
In the bipartitions $((1^3),(1^2))$ and $((3),(1^2))$ there is no border strip of size $5$ at all. Thus, the formula tells us that the corresponding character values are $0$. On the other hand, the bipartition $((3,1^2),\emptyset)$ consists of a single border strip of size $5$ and height $2$. The formula gives 
$$\chi_{(3,1^2),\emptyset}(\emptyset,(5)) = \chi_{\emptyset} = 1.$$
Putting things together, we deduce that $\chi_{(3,1^2),(4,2)}((4),(5,2)) = -1$.

\paragraph{}\label{newlabel}We may now describe the horizontal arrows in \ref{comparisontheorem} Theorem for the unitary group. To do this, we need an alternate labelling of the irreducible unipotent representations of the unitary group. We refer to \cite{fong} for the details.\\
The new labelling of the irreducible unipotent representations of $\mathrm U_n(q)$ involves triples of the form $(\Delta_x,\alpha,\beta)$ where $x$ is a nonnegative integer such that $n = 2a + \frac{x(x+1)}{2}$ for some integer $a\geq 0$, and where $(\alpha,\beta)$ is a bipartition of $a$. The corresponding representation will be denoted $\rho_{\Delta_x,\alpha,\beta}$. With this labelling, the unipotent Harish-Chandra series $\mathcal E(\mathrm U_n(q),(L_x,\rho_x))$ consists precisely of all the representations $\rho_{\Delta_x,\alpha,\beta}$ with $(\alpha,\beta)$ varying over all bipartitions of $a$. The bijection $\mathbb Z\mathcal E(\mathrm U_n(q),(L_x,\rho_x)) \xrightarrow{\sim} \mathbb Z\mathrm{Irr}(W_{\mathrm U_n(q)}(L_x,\rho_x))$ involved in the Comparison theorem simply sends $\rho_{\Delta_x,\alpha,\beta}$ to $\chi_{\alpha,\beta}$. Here, we made use of the identification $W_{\mathrm U_n(q)}(L_x,\rho_x) \simeq W_a$ as in \ref{unitaryrelativeWeylgroup}.\\
More generally, if $M$ is a standard Levi complement in $\mathrm U_n(q)$ containing $L_x$, we may write $M \simeq \mathrm U_b(q)\times \mathrm{GL}_{a_1}(q^2) \times \ldots \times \mathrm{GL}_{a_r}(q^2)$ where $n = 2(a_1 + \ldots + a_r) + b$ and $b\geq \frac{x(x+1)}{2}$. The irreducible unipotent representations of $M$ in the Harish-Chandra series $\mathcal E(M,(L_x,\rho_x))$ are those of the form $\rho_{\Delta_x,\alpha,\beta}\boxtimes \rho^{\mathrm{GL}}_{\lambda_1}\boxtimes \ldots \boxtimes \rho^{\mathrm{GL}}_{\lambda_r}$ where $\lambda_i$ is a partition of $a_i$ for $1\leq i \leq r$ and $(\alpha,\beta)$ is a bipartition of the integer $c:= \frac{1}{2}\left(b - \frac{x(x+1)}{2}\right)$. On the other hand, the relative Weyl group $W_M(L_x,\rho_x)$ can be identified with the subgroup of $W_{\mathrm U_n(q)}(L_x,\rho_x) \simeq W_a$ isomorphic to the product $W_c\times \mathfrak S_{a_1}\times \ldots \times \mathfrak S_{a_r}$ (note that $c + a_1 + \ldots + a_r = a$). With the notations of \ref{unitaryrelativeWeylgroup}, the $W_c$-component is generated by the elements $\theta,\sigma_1,\ldots \sigma_{c-1}$, the $\mathfrak S_{a_1}$-component by the elements $\sigma_{c+1}\ldots ,\sigma_{c+a_1-1}$, and so on. Irreducible characters of $W_M(L_x,\rho_x)$ have the shape $\chi_{\alpha,\beta}\boxtimes \chi_{\lambda_1} \boxtimes \ldots \boxtimes \chi_{\lambda_r}$ where $(\alpha,\beta)$ is a bipartition of $c$ and $\lambda_i$ is a partition of $a_i$ for $1\leq i \leq r$.\\
Then, according to \cite{fong} (4.2), the bijection $\mathbb Z\mathcal E(\mathrm M,(L_x,\rho_x)) \xrightarrow{\sim} \mathbb Z\mathrm{Irr}(W_M(L_x,\rho_x))$ involved in the Comparison theorem in \ref{comparisontheorem} sends $\rho_{\Delta_x,\alpha,\beta}\boxtimes \rho^{\mathrm{GL}}_{\lambda_1}\boxtimes \ldots \boxtimes \rho^{\mathrm{GL}}_{\lambda_r}$ to $\chi_{\alpha,\beta}\boxtimes \chi_{\lambda_1} \boxtimes \ldots \boxtimes \chi_{\lambda_r}$. 

\paragraph{}\label{twolabels} We explain how the two different labellings of the irreducible unipotent representations of $\mathrm{U}_n(q)$ are related. To do this, one needs the notion of $2$-quotient. For the following definitions, we allow partitions to have $0$ terms at the end. Thus, let us write $\lambda = (\lambda_1 \geq \ldots \geq \lambda_r)$ with $\lambda_r \geq 0$. The \textbf{$\beta$-set} of $\lambda$ is the sequence of decreasing nonnegative integers $\beta_i := \lambda_i + r - i$ for $1\leq i \leq r$. Mapping a partition $\lambda$ to its $\beta$-set gives a bijection between the set of partitions having $r$ terms and the set of decreasing sequences of nonnegative integers of length $r$. The inverse mapping sends a sequence $(\beta_1> \ldots > \beta_r \geq 0)$ to the partition $\lambda$ given by $\lambda_i = \beta_i + i - r$.\\
Let $\lambda$ be a partition of $n$ as above, and let $\beta$ be its $\beta$-set. We let $\beta_{\text{even}}$ (resp. $\beta_{\text{odd}}$) be the subsequence consisting of all even (resp. odd) integers of $\beta$. Then, we define the following sequences.
\begin{align*}
\beta^0 := \left(\frac{\beta_i}{2} \,\middle|\, \beta_i \in \beta_{\text{even}}\right) & & \beta^1 := \left(\frac{\beta_i - 1}{2} \,\middle|\, \beta_i \in \beta_{\text{odd}}\right) 
\end{align*}
The sequences $\beta^0$ and $\beta^1$ are the $\beta$-sets of two partitions, which we call $\mu^0$ and $\mu^1$ respectively. Then, the \textbf{$2$-quotient} of $\lambda$ is the bipartition $(\mu^0,\mu^1)$ if $r$ is odd, and $(\mu^1,\mu^0)$ if $r$ is even. We note that the ordering of $\mu^0$ and $\mu^1$ in the $2$-quotient may vary in the literature. Here, we followed the conventions of \cite{fong} section 1. A different ordering is used in \cite{james} 2.7.29. In loc. cit. Theorem 2.7.37, another construction of the $2$-quotient using Young diagrams is proposed.\\
Let $\lambda'$ be another partition which differs from $\lambda$ only by $0$ terms at the end. While the $\beta$-sets of $\lambda$ and $\lambda'$ are not the same, the resulting $2$-quotients are equal up to $0$ terms at the end of the partitions. Thus, from now on we identify all partitions differing only from $0$ terms by removing all of them. The $2$-quotient of a partition is then well-defined.

\noindent \begin{theo}[\cite{james} Theorem 2.7.30]
A partition $\lambda$ is uniquely characterized by the data of its $2$-core $\Delta_x$ and its $2$-quotient $(\lambda^0,\lambda^1)$. Moreover, the lengths of these partitions are related by the equation 
$$|\lambda| = |\Delta_x| +2(|\lambda^0| + |\lambda^1|)$$
and $|\Delta_x| = \frac{x(x+1)}{2}$.
\end{theo}

\noindent For instance, the $2$-quotient of the partition $\lambda = (3^2,2^2,1)$ is $(2^2,1)$. Recall that the $2$-core of $\lambda$ is $\Delta_1$. Thus, the equation on the lengths of the partitions is satisfied, as we have $11 = 1 + 2(4+1)$.\\
We may now relate the two labellings $\{\rho^{\mathrm U}_{\lambda}\}$ and $\{\rho_{\Delta_x,\alpha,\beta}\}$ of the irreducible unipotent representations of $\mathrm U_n(q)$ together. 

\noindent \begin{prop}[\cite{fong} Appendix]
Let $\lambda$ be a partition of $n$. Denote by $\Delta_{y}$ its $2$-core and by $(\lambda^0,\lambda^1)$ its $2$-quotient. On the other hand, let $x\geq 0$ be such that $n = 2a + \frac{x(x+1)}{2}$ for some $a\geq 0$ and let $(\alpha,\beta)$ be a bipartition of $a$. Then the irreducible representations $\rho^{U}_{\lambda}$ and $\rho_{\Delta_{x},\alpha,\beta}$ are equivalent if and only if $x=y$ and $(\lambda^0,\lambda^1) = (\alpha,\beta)$ if $x$ is even or $(\lambda^0,\lambda^1) = (\beta,\alpha)$ if $x$ is odd.
\end{prop}

\noindent For instance, for $\lambda = (3^2,2^2,1)$ the representation $\rho_{\lambda}^{\mathrm{U}}$ is equivalent to $\rho_{\Delta_1,(1),(2^2)}$.

\paragraph{}\label{Pierirule} In order to apply the comparison theorem \ref{comparisontheorem} for unitary groups, it remains to understand how to compute inductions in Coxeter groups of type $B$. Such computations are carried out in \cite{geck2} Section 6.1. It turns out that we will only need one specific case of such inductions, and the corresponding method is known as the Pieri rule for groups of type $B$.

\noindent \begin{prop}[\cite{geck2} 6.1.9] Let $a \geq 1$ and consider $r,s \geq 0$ such that $r+s = a$. We think of the group $W_r \times \mathfrak S_s$ as a subgroup of $W_a$ as in \ref{newlabel}.
\begin{enumerate}[label={--},noitemsep,topsep=0pt]
\item Let $(\alpha,\beta)$ be a bipartition of $r$. Then the induced character 
$$\mathrm{Ind}_{W_r \times \mathfrak S_s}^{W_a}\left(\chi_{(\alpha,\beta)} \boxtimes \chi_{(s)} \right)$$
is the multiplicity-free sum of all the characters $\chi_{\gamma,\delta}$ such that for some $0\leq k \leq s$, the Young diagram of $\gamma$ (resp. $\delta$) can be obtained from that of $\alpha$ (resp. $\beta$) by adding $k$ boxes (resp. $s-k$ boxes) so that no two of them lie in the same column.

\item Let $(\gamma,\delta)$ be a bipartition of $a$. The restricted character 
$$\mathrm{Res}_{W_r}^{W_a}\left(\chi_{\gamma,\delta}\right)$$
is the multiplicity-free sum of all the characters $\chi_{(\alpha,\beta)}$ such that for some $0\leq k \leq s$, the Young diagram of $\alpha$ (resp. $\beta$) can be obtained from that of $\gamma$ (resp. $\delta$) by deleting $k$ boxes (resp. $s-k$ boxes) so that no two of them lie in the same column.
\end{enumerate}
\end{prop}

\noindent We will use this rule on concrete examples in the sections that follow.

\section{The cohomology of the Coxeter variety for the unitary group}

\paragraph{}\label{Lusztigresults} In this section, we describe the cohomology of the Coxeter varieties for the unitary groups in odd dimension in terms of the classification of unipotent representations that we recalled in the previous section. The cohomology groups are entirely understood by the work of Lusztig in \cite{cox}.\\
Let $t\geq 0$. The \textbf{Coxeter variety} for $\mathrm U_{2t+1}(q)$ is the Deligne-Lusztig variety $X_{\emptyset}(\mathrm{cox})$, where $\mathrm{cox}$ is any Coxeter element of the Weyl group $\mathbf W \simeq \mathfrak S_{2t+1}$. Recall that a Coxeter element is a permutation which can be written as the product, in any order, of exactly one simple reflexion for each $F$-orbit on $\mathbf S$. The variety $X_{\emptyset}(\mathrm{cox})$ does not depend on the choice of the Coxeter element. It is defined over $\mathbb F_{q^2}$ and is equipped with commuting actions of both $\mathrm U_{2t+1}(q)$ and $F^2$.

\begin{notation}
We write $X^t = X_{\emptyset}(\mathrm{cox})$ for the Coxeter variety attached to the unitary group $\mathrm U_{2t+1}(q)$. We also write $\mathrm{H}_c^{\bullet}(X^t)$ instead of $\mathrm{H}_c^{\bullet}(X_{\emptyset}(\mathrm{cox})\otimes \mathbb F,\overline{\mathbb Q_{\ell}})$, where $\ell \not = p$.
\end{notation}

\noindent We first recall known facts on the cohomology of $X^t$ from Lusztig's work.

\noindent \begin{theo}[\cite{cox}]
The following statements hold.
\begin{enumerate}[label=\upshape (\arabic*), topsep = 0pt]
\item The variety $X^t$ has dimension $t$ and is affine. The cohomology group $\mathrm{H}^{t+i}_c(X^t)$ is zero unless $0\leq i \leq t$.
\item The Frobenius $F^2$ acts in a semisimple manner on the cohomology of $X^t$. 
\item The group $\mathrm{H}^{2t}_c(X^t)$ is $1$-dimensional, the unitary group $\mathrm U_{2t+1}(q)$ acts trivially whereas $F^2$ has a single eigenvalue $q^{2t}$.
\item The group $\mathrm H^{t+i}_c(X^t)$ for $0 \leq i < t$ is the direct sum of two eigenspaces of $F^2$, for the eigenvalues $q^{2i}$ and $-q^{2i+1}$. Each eigenspace is an irreducible unipotent representation of $\mathrm{U}_{2t+1}(q)$.
\item If $0 \leq a \leq 2t$, the dimension of the eigenspace of $(-q)^a$ inside the sum $\sum_{i\geq 0} \mathrm H_c^{t+i}(X^t)$ is given by the formula 
$$q^{\frac{(2t-a)(2t+1-a)}{2}}\prod_{j = 1}^{2t-a} \frac{q^{a + j} - (-1)^{a+j}}{q^{j} - (-1)^{j}}.$$
\item The sum $\sum_{i\geq 0} \mathrm H_c^{t+i}(X^t)$ is multiplicity-free as a representation of $\mathrm U_{2t+1}(q)$.
\end{enumerate}
\end{theo}

\paragraph{}\label{2-core2-quotient} We wish to identify these unipotent representations of $\mathrm U_{2t+1}(q)$ occuring in the cohomology of $X^t$. To this purpose, we start by defining the following partitions. If $0\leq a \leq 2t$, we put $\lambda_a^t  := (1 + a, 1^{2t - a})$. Note that $\lambda_0^t = (1^{2t+1})$ and $\lambda_{2t}^t = (2t+1)$. 

\noindent \begin{lem}
For $0 \leq i \leq t$, the $2$-core of $\lambda_{2i}^t$ is $\Delta_1$ and its $2$-quotient is $((1^{t-i}),(i))$.\\
For $0 \leq i < t$, then the $2$-core of $\lambda_{2i+1}^t$ is $\Delta_2$ and its $2$-quotient is $((i),(1^{t-i-1}))$.
\end{lem}

\noindent In particular, according to \ref{twolabels} the irreducible unipotent representation $\rho_{\lambda_{2i}^t}$ of $\mathrm U_{2t+1}(q)$ is equivalent to the representation $\rho_{\Delta_1,(i),(1^{t-i})}$, and $\rho_{\lambda_{2i+1}^t}$ to $\rho_{\Delta_2,(i),(1^{t-i-1})}$.

\begin{proof}
The Young diagram of the partition $\lambda_a^t$ has the following shape. 
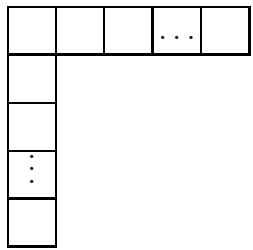
\begin{center}
\ydiagram{3,1,1}*[\dots]{3+1}*[\vdots]{0,0,0,1}*{4+1,0,0,0,1}
\end{center}
The first row has an odd number of boxes when $a$ is even, and an even number of boxes when $a$ is odd. To compute the $2$-core, one removes horizontal dominoes from the first row, right to left, and vertical dominoes from the first column, bottom to top. The process results in $\Delta_1$ when $a$ is even and $\Delta_2$ when $a$ is odd. \\
The partition $\lambda_a^t$ has $2t + 1 - a$ non zero terms. Its $\beta$-set is given by the sequence $$\beta = (2t+1,2t-a,2t-a-1,\ldots ,1).$$
Assume that $a = 2i$ is even. Then the sequences $\beta^0$ and $\beta^1$ are given by 
\begin{align*}
\beta^0 = (t-i,t-i-1,\ldots 1), & & \beta^1 = (t,t-i-1,t-i-2,\ldots ,0).
\end{align*}
The sequence $\beta^0$ has length $t-i$ while $\beta^1$ has length $t-i+1$. The associated permutations are then respectively $\mu_0 = (1^{t-i})$ and $\mu_1 = (i)$. Since $2t + 1 - a$ is odd, the $2$-quotient is given by $(\mu_0,\mu_1)$ as claimed.\\
Assume now that $a = 2i+1$ is odd. Then the sequences $\beta^0$ and $\beta^1$ are given by 
\begin{align*}
\beta^0 = (t-i-1,t-i-2,\ldots 1), & & \beta^1 = (t,t-i-1,t-i-2,\ldots ,0).
\end{align*}
The sequence $\beta^0$ has length $t-i-1$ while $\beta^1$ has length $t-i+1$. The associated permutations are then respectively $\mu_0 = (1^{t-i-1})$ and $\mu_1 = (i)$. Since $2t + 1 - a$ is even, the $2$-quotient is given by $(\mu_1,\mu_0)$ as claimed.
\end{proof}

\paragraph{}We may now identify the irreducible unipotent representations occuring in the cohomology of the Coxeter variety $X^k$.

\label{cohomologyCoxeter} \begin{prop}
For $0 \leq i < t$, the cohomology group of the Coxeter variety for the finite unitary group $\mathrm U_{2t+1}(q)$ is given by
$$\mathrm H_c^{t+i}(X^t) = \rho_{\lambda_{2i}^t} \oplus \rho_{\lambda_{2i+1}^t}$$
with the first summand corresponding to the eigenvalue $q^{2i}$ of $F^2$ and the second to $-q^{2i+1}$. Moreover, $\mathrm H_c^{2t}(X^t) = \rho_{\lambda_{2t}^t}$ with eigenvalue $q^{2t}$.
\end{prop}

\noindent Before going to the proof, one may notice that the statement is consistent with the dimensions. Indeed, the formula given in \ref{Lusztigresults} Theorem (5) coincides with the hook formula for the degree of the representation $\rho^{\mathrm U}_{\lambda^{t}_a}$ given in \ref{dimensionformula} Proposition.

\begin{proof} First, the statement on the highest cohomology group $\mathrm H_c^{2t}(X^t)$ follows from \ref{Lusztigresults} Theorem (3). It is the only cohomology group in the case $t=0$. We will prove the formula by induction on $t$. Let us now assume $t\geq 1$ and that the proposition is known for $t-1$. If $0\leq i \leq t-1$, we know that $\mathrm H_c^{t+i}(X^t)$ is the sum of two irreducible unipotent representations. So let us write 
$$\mathrm H_c^{t+i}(X^t) = \rho_{\mu_i} \oplus \rho_{\nu_i}$$
where $\mu_i$ and $\nu_i$ are two partitions of $2t+1$, and so that $\rho_{\mu_i}$ corresponds to the eigenvalue $q^{2i}$ of $F^2$ whereas $\rho_{\nu_i}$ corresponds to $-q^{2i+1}$.\\
We consider the standard Levi complement $L \simeq \mathrm{GL}_1(q^2) \times \mathrm U_{2t-1}(q) \subset \mathrm U_{2t+1}(q)$. Let $V$ denote the unipotent radical of the standard parabolic subgroup containing $L$. According to \cite{cox} Corollary 2.10, one can build a geometric isomorphism between the quotient variety $X^t/V$ and the product of the Coxeter variety for $L$ and of a copy of $\mathbb G_m$. Even though this geometric isomorphism is not $L$-equivariant, Lusztig proves that the induced map on cohomology is $L$-equivariant. By a discussion similar to that in \ref{decompositionDLvarieties}, the Coxeter variety for $L$ is isomorphic to the Coxeter variety $X^{t-1}$ for $\mathrm U_{2t-1}(q)$. We write ${}^*\mathrm R_{t-1}^{t}$ for the composition of the Harish-Chandra restriction from $\mathrm U_{2t+1}(q)$ to $L$, with the usual restriction from $L$ to the subgroup $\mathrm U_{2t-1}(q)$.  For any nonnegative integer $i$, the $\mathrm{U}_{2t-1}(q), F^2$-equivariant induced map on the cohomology is an isomorphism
\begin{equation}\label{eq}
{}^*\mathrm R_{t-1}^{t} \left(\mathrm H^{t+i}_c(X^t)\right) \simeq \mathrm H^{t-1+i}_c(X^{t-1}) \oplus \mathrm H^{t-1+(i-1)}_c(X^{t-1})(1)\tag{$**$}.
\end{equation}
Here, $(1)$ denotes the Tate twist (the action of $F^2$ on a twist $M(n)$ is obtained from the action on the space $M$ by multiplication with $q^{2n}$). The right-hand side of this identity is given by the induction hypothesis. Let us look at the left-hand side.\\
We fix $0\leq i \leq t-1$ and we denote by $(\Delta_x,\alpha,\beta)$ and by $(\Delta_y,\gamma,\delta)$ the alternative labelling of the representations $\rho_{\mu_i}$ and $\rho_{\nu_i}$ respectively as introduced in \ref{newlabel} and \ref{twolabels}. By the Howlett-Lehrer comparison theorem for restriction in \ref{comparisontheorem} and by the Pieri rule in \ref{Pierirule}, we know that the restriction ${}^*\mathrm R_{t-1}^{t} \left(\rho_{\Delta_x,\alpha,\beta}\right)$ is the multiplicity-free sum of all the representations $\rho_{\Delta_x,\alpha',\beta'}$ where the bipartition $(\alpha',\beta')$ can be obtained from $(\alpha,\beta)$ by removing exactly one box, of either $\alpha$ or $\beta$. The similar description also holds for ${}^*\mathrm R_{t-1}^{t} \left(\rho_{\Delta_y,\gamma,\delta}\right)$.\\
By using \ref{2-core2-quotient} Lemma and the induction hypothesis, we may write down the identity \eqref{eq} explicitely. Moreover, as it is $F^2$-equivariant we can identify the components corresponding to the same eigenvalues on both sides. We distinguish $4$ different cases depending on the values of $t$ and $i$.
\begin{enumerate}[label={--},noitemsep,topsep=0pt]
\item \textbf{Case }$\mathbf{t=1}$. We only need to consider $i=0$. On the right-hand side of \eqref{eq}, the second term is $0$ because $t-1 + (i-1) = -1 < 0$. On the other hand, the first term is $\rho_{\lambda^0_0} \simeq \rho_{\Delta_1,\emptyset, \emptyset}$ and it corresponds to the eigenvalue $(-q)^0 = 1$. By identifying the eigenspaces, we have ${}^*\mathrm R_{0}^{1} \left(\rho_{\Delta_x,\alpha,\beta}\right) \simeq \rho_{\Delta_1,\emptyset,\emptyset}$ and ${}^*\mathrm R_{0}^{1} \left(\rho_{\Delta_y,\gamma,\delta}\right) = 0$. The second equation implies that there is no box to remove from $\gamma$ nor from $\delta$. Thus, $\gamma = \delta = \emptyset$. The value of $y$ is given by the relation $2t+1 = 3 = 2(0+0) + \frac{y(y+1)}{2}$, that is $y=2$. This corresponds to the partition $\nu_0 = \lambda^{1}_1$. We notice in passing that the representation $\rho_{\nu_0}$ is the unique unipotent cuspidal representation of $\mathrm U_3(q)$.\\
As for $\mu_0$, the equation ${}^*\mathrm R_{0}^{1} \left(\rho_{\Delta_x,\alpha,\beta}\right) \simeq \rho_{\Delta_1,\emptyset,\emptyset}$ tells us that there is only one removable box from $(\alpha,\beta)$. After removal of this box, both partitions are empty. Thus, we deduce that $x=1$ and $(\alpha,\beta) = (1,\emptyset) \text{ or } (\emptyset,1)$. This corresponds respectively to $\mu_0 = \lambda^1_2$ or $\mu_0 = \lambda^1_0$. That is, $\rho_{\mu_0}$ is either the trivial or the Steinberg representation of $\mathrm U_3(q)$. We can deduce which one it is by comparing the degree of the representations with the formula of \ref{Lusztigresults} Theorem (5). According to this formula, the dimension of the eigenspace for $(-q)^0$ is $q^3$. This is precisely the degree of the Steinberg representation $\rho_{\lambda^{1}_0}$ as given by the hook formula in \ref{dimensionformula} Proposition, and it excludes the possibility of $\rho_{\mu_0}$ being trivial. Thus, we have $\mu_0 = \lambda^1_0$ as claimed.
\end{enumerate}

\noindent From now, we assume $t\geq 2$.

\begin{enumerate}[label={--},noitemsep,topsep=0pt]
\item \textbf{Case }$\mathbf{i=0}$. On the right-hand side of \eqref{eq}, the second term is $0$ because $t-1 + (i-1) = t-2 < t-1$. The first term is $\rho_{\lambda^{t-1}_0} \oplus \rho_{\lambda^{t-1}_1} \simeq \rho_{\Delta_1,\emptyset,(1^{t-1})} \oplus \rho_{\Delta_2,\emptyset,(1^{t-2})}$. Identifying the eigenspaces, we have ${}^*\mathrm R_{t-1}^{t} \left(\rho_{\Delta_x,\alpha,\beta}\right) \simeq \rho_{\Delta_1,\emptyset,(1^{t-1})}$ and ${}^*\mathrm R_{t-1}^{t} \left(\rho_{\Delta_y,\gamma,\delta}\right) \simeq \rho_{\Delta_2,\emptyset,(1^{t-2})}$. We deduce that $x=1$ and $y=2$. Moreover, it also follows that there is only one removable box in $(\alpha,\beta)$ and in $(\gamma,\delta)$. After removal, we should obtain respectively $(\emptyset,(1^{t-1}))$ and $(\emptyset,(1^{t-2}))$. The only possibility is that $(\alpha,\beta) = (\emptyset,(1^t))$ and $(\gamma,\delta) = (\emptyset,(1^{t-1}))$. This corresponds to $\mu_0 = \lambda^{t}_0$ and $\nu_0 = \lambda^{t}_1$ as claimed.

\item \textbf{Case }$\mathbf{i=t-1}$. On the right-hand side of \eqref{eq}, the first term is $\rho_{\lambda^{t-1}_{2(t-1)}}\simeq \rho_{\Delta_1,(t-1),\emptyset}$ and the second term is $\rho_{\lambda^{t-1}_{2(t-2)}} \oplus \rho_{\lambda^{t-1}_{2(t-2)+1}} \simeq \rho_{\Delta_1,(t-2),(1)} \oplus \rho_{\Delta_2,(t-2),\emptyset}$. Identifying the eigenspaces while taking the Tate twist into account, we have ${}^*\mathrm R_{t-1}^{t} \left(\rho_{\Delta_x,\alpha,\beta}\right) \simeq \rho_{\Delta_1,(t-1),\emptyset} \oplus \rho_{\Delta_1,(t-2),(1)}$ and ${}^*\mathrm R_{t-1}^{t} \left(\rho_{\Delta_y,\gamma,\delta}\right) \simeq \rho_{\Delta_2,(t-2),\emptyset}$. We deduce that $x=1$ and $y=2$. Moreover, there are two removable boxes in $(\alpha,\beta)$ and only one removable box in $(\gamma,\delta)$. After removal of one of the two boxes in $(\alpha,\beta)$, we can get either $((t-1),\emptyset)$ or $((t-2),(1))$ ; and after removal of the box in $(\gamma,\delta)$ we obtain $((t-2),\emptyset)$. The only possibility is that $(\alpha,\beta) = ((t-1),(1))$ and $(\gamma,\delta) = ((t-1),\emptyset)$. This corresponds to $\mu_{t-1} = \lambda^t_{2(t-1)}$ and $\nu_{t-1} = \lambda^{t}_{2(t-1)+1}$ as claimed. 

\item \textbf{Case }$\mathbf{1\leq i \leq t-2}$. On the right-hand side of \eqref{eq}, the first term is $\rho_{\lambda^{t-1}_{2i}} \oplus \rho_{\lambda^{t-1}_{2i+1}} \simeq \rho_{\Delta_1,(i),(1^{t-1-i})} \oplus \rho_{\Delta_2,(i),(1^{t-2-i})}$. The second term is $\rho_{\lambda^{t-1}_{2(i-1)}} \oplus \rho_{\lambda^{t-1}_{2(i-1)+1}} \simeq \rho_{\Delta_1,(i-1),(1^{t-i})} \oplus \rho_{\Delta_2,(i-1),(1^{t-1-i})}$. Identifying the eigenspaces while taking the Tate twist into account, we have ${}^*\mathrm R_{t-1}^{t} \left(\rho_{\Delta_x,\alpha,\beta}\right) \simeq \rho_{\Delta_1,(i),(1^{t-1-i})} \oplus \rho_{\Delta_1,(i-1),(1^{t-i})}$ and ${}^*\mathrm R_{t-1}^{t} \left(\rho_{\Delta_y,\gamma,\delta}\right) \simeq \rho_{\Delta_2,(i),(1^{t-2-i})} \oplus \rho_{\Delta_2,(i-1),(1^{t-1-i})}$. We deduce that $x = 1$ and $y = 2$. Moreover, there are exactly two removable boxes from $(\alpha,\beta)$ and from $(\gamma,\delta)$. After removal of one of the two boxes in $(\alpha,\beta)$, we can get either $((i),(1^{t-1-i}))$ or $((i-1),(1^{t-i}))$ ; and after removal of one of the two boxes in $(\gamma,\delta)$, we can get either $((i),(1^{t-2-i}))$ or $((i-1),(1^{t-1-i}))$. The only possibility is that $(\alpha,\beta) = ((i),(1^{t-i}))$ and $(\gamma,\delta) = ((i),(1^{t-1-i}))$. This corresponds to $\mu_i = \lambda^t_{2i}$ and $\nu_i = \lambda^t_{2i+1}$ as claimed. 
\end{enumerate}
\end{proof}

\section{The cohomology of the variety $X_I(\mathrm{id})$}

\paragraph{}We go on with the computation of the cohomology of the variety $X_I(\mathrm{id})$. We use the same notations as in section 1. We first compute the cohomology of each Ekedahl-Oort stratum $X_{I_t}(w_t)$, before using the spectral sequence associated to the stratification to conclude.\\
Recall that $X_I(\mathrm{id})$ has dimension $d$, is defined over $\mathbb F_{q^2}$ and is equipped with an action of $J \simeq \mathrm{U}_{2d+1}(q)$. As before, we will write $\mathrm{H}_c^{\bullet}(X_I(\mathrm{id}))$ as a shortcut for $\mathrm{H}_c^{\bullet}(X_I(\mathrm{id})\otimes \mathbb F,\overline{\mathbb Q_{\ell}})$.

\noindent \begin{theo} 
The following statements hold. 
\begin{enumerate}[label=\upshape (\arabic*), topsep = 0pt]
		\item The cohomology group $\mathrm H_c^i(X_I(\mathrm{id}))$ is zero unless $0 \leq i \leq 2d$. There is an isomorphism $\mathrm H_c^i(X_I(\mathrm{id})) \simeq \mathrm H_c^{2d-i}(X_I(\mathrm{id}))^{\vee}(d)$ which is equivariant for the actions of $F^2$ and of $\mathrm U_{2d+1}(q)$. 
		\item The Frobenius $F^2$ acts like multiplication by $(-q)^i$ on $\mathrm H_c^i(X_I(\mathrm{id}))$. 
		\item For $0\leq i \leq d$ we have 
		$$\mathrm H_c^{2i}(X_I(\mathrm{id})) = \bigoplus_{s=0}^{\min(i,d - i)} \rho_{(2d + 1 - 2s, 2s)}.$$
		For $0\leq i \leq d - 1$ we have 
		$$\mathrm H_c^{2i+1}(X_I(\mathrm{id})) = \bigoplus_{s=0}^{\min(i,d - 1 - i)} \rho_{(2d - 2s, 2s + 1)}.$$
	\end{enumerate}
\end{theo}

\noindent Thus, in the cohomology of $X_I(\mathrm{id})$ all the representations associated to a Young diagram with at most $2$ rows occur, and there is no other. Such a diagram has the following general shape.
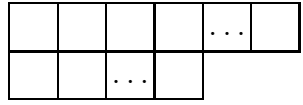
\begin{center}
\ydiagram{4,2}*[\ldots]{4+1,2+1}*{5+1,3+1}
\end{center}
We may rephrase the result by using the alternative labelling of the irreducible unipotent representations as in \ref{twolabels}. The partition $(2d+1-2s,2s)$ has $2$-core $\Delta_1$ and $2$-quotient $(\emptyset,(d-s,s))$ ; whereas the partition $(2d-2s,2s+1)$ has $2$-core $\Delta_2$ and $2$-quotient $((d-1-s,s),\emptyset)$. Thus, according to \ref{twolabels} Proposition, we have 
\begin{align*}
\rho_{(2d + 1 - 2s, 2s)} \simeq \rho_{\Delta_1,(d-s,s),\emptyset}, & & \rho_{(2d - 2s, 2s + 1)} \simeq \rho_{\Delta_2,(d-1-s,s),\emptyset}.
\end{align*}
In particular, all irreducible representations in the cohomology groups of even index belong to the unipotent principal series $\mathcal E(\mathrm U_{2d +1}(q),(L_1,\rho_1))$, whereas all the ones in the groups of odd index belong to the Harish-Chandra series $\mathcal E(\mathrm U_{2d+1}(q),(L_2,\rho_2))$.

\begin{proof}
Point (1) of the statement follows from a general property of the cohomology groups, namely Poincaré duality. It is due to the fact that $X_I(\mathrm{id})$ is projective and smooth. It also implies the purity of the Frobenius $F^2$ on the cohomology : we know at this stage that all eigenvalues of $F^2$ on $\mathrm H_c^i(X_I(\mathrm{id}))$ have complex modulus $q^i$ under any choice of an isomorphism $\overline{\mathbb Q_{\ell}} \simeq \mathbb C$.\\
We prove the points (2) and (3) by explicit computations. As in \ref{2-core2-quotient}, we denote by $\lambda^t_a$ the partition $(1+a, 1^{2t-a})$ of $2t+1$. Let $0\leq t \leq d$. For $0\leq a \leq 2t$ we will write 
$$\mathrm R_{a}^{t} := \mathrm R_{L_{K_{t}}}^{\mathrm U_{2d+1}(q)}\left(\rho_{(d-t)}^{\mathrm{GL}}\boxtimes \rho^{\mathrm U}_{\lambda^{t}_{a}}\right).$$ 
Recall that \ref{isomorphismEOstratum} Proposition gives an isomorphism between the Ekedahl-Oort stratum $X_{I_t}(w_t)$ and the variety $\mathrm U_{2d+1}(q)/U_{K_{t}} \times_{L_{K_{t}}} X^{\mathbf L_{K_{t}}}_{I_{t}}(w_{t})$. It implies that the cohomology of the Ekedahl-Oort stratum is the Harish-Chandra induction of the cohomology of the Deligne-Lusztig variety $X_{I_{t}}^{\mathbf L_{K_{t}}}(w_{t})$. According to \ref{decompositionDLvarieties}, this cohomology is related to that of the Coxeter variety for $\mathrm U_{2t + 1}(q)$. Combining with the formula of \ref{cohomologyCoxeter} Proposition, for $0 \leq i \leq t - 1$ it follows that 
\begin{align*}
\mathrm H_c^{t + i}(X_{I_t}(w_t)) = \mathrm R_{2i}^{t} \oplus \mathrm R_{2i+1}^{t}, & & \mathrm H_c^{2t}(X_{I_t}(w_t)) = \mathrm R_{2t}^{t}.
\end{align*}
The representation $\mathrm R_{a}^{t}$ in this formula is associated to the eigenvalue $(-q)^a$ of $F^2$.\\

\noindent We first compute $\mathrm R_a^{t}$ explicitely. By the combination of the Howhlett-Lehrer comparison theorem in \ref{comparisontheorem} and the Pieri rule for groups of type $B$ as in \ref{Pierirule}, one can compute the Harish-Chandra induction $\mathrm R_a^{t}$ by adding $d-t$ boxes to the bipartition corresponding to the representation $\rho^{\mathrm U}_{\lambda^{t}_a}$ with no two added boxes in the same column. Recall from \ref{2-core2-quotient} Lemma that the representation $\rho_{\lambda_{2i}^{t}}$ of $\mathrm U_{2t+1}(q)$ is equivalent to the representation $\rho_{\Delta_1,(i),(1^{t-i})}$, and that $\rho_{\lambda_{2i+1}^{t}}$ is equivalent to $\rho_{\Delta_2,(i),(1^{t-1-i})}$.\\
In order to illustrate the argument, let us say that we want to add $N$ boxes to a bipartition of the shape as in the figure below, so that no two added boxes lie in the same column. 
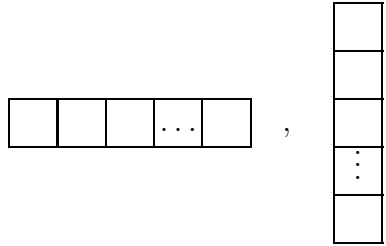
\begin{center}
\ydiagram{3}*[\ldots]{3+1}*{4+1} \quad , \quad \ydiagram{1,1,1,0,1}*[\vdots]{0,0,0,1}
\end{center}
We will add $N_1$ boxes to the first diagram and $N_2$ to the second, where $N = N_1 + N_2$. In the first diagram, the only places where we can add boxes are in the second row from left to right, and at the end of the first row. Because no two added boxes must be in the same column, the number of boxes we add on the second row must be at most the number of boxes already lying in the first row. Of course, it must also be at most $N_1$.\\
In the second diagram, the only places where we can add boxes are at the bottom of the first column and at the end of the first row. Because no two added boxes must be in the same column, we can only put up to one box at the bottom of the first column and all the remaining ones will align at the end of the first row.\\
At the end of the process, we will obtain a bipartition of the following general shape. 
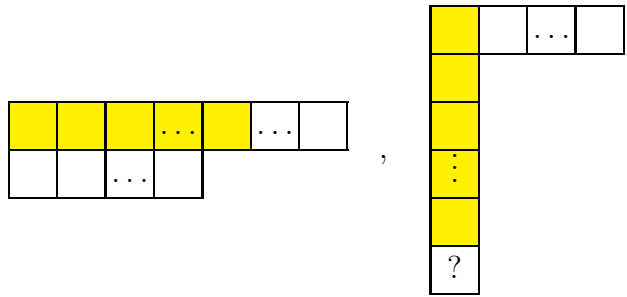
\begin{center}
\ydiagram[*(yellow)]{3}*{0,2}*[*(yellow) \dots]{3+1}*[\dots]{0,2+1}*[*(yellow)]{4+1}*{0,3+1}*[\ldots]{5+1}*{6+1} \quad , \quad \ydiagram[*(yellow)]{1,1,1,0,1}*{1+1}*[*(yellow) \vdots]{0,0,0,1}*[\dots]{2+1}*[?]{0,0,0,0,0,1}*{3+1}
\end{center}
We colored in yellow the boxes that were already there before we added new ones. The box with a question mark may or may not be placed there.\\
We now make the result more precise, and write down exactly what the irreducible components of $\mathrm R_{a}^{t}$ are depending on the parity of $a$.

\begin{enumerate}[label={--},noitemsep,topsep=0pt]
\item For $0\leq i \leq t$, the representation $\mathrm R_{2i}^{t}$ is the multiplicity-free sum of all the representations $\rho_{\Delta_1,\alpha,\beta}$ where the bipartition $(\alpha,\beta)$ satisfies, for some $0 \leq x \leq d - t$,
$$\begin{cases}
\alpha = (i + x - s, s) \text{ for some } 0\leq s \leq \min(x,i),\\
\beta = (d - t - x, 1^{t - i}) \text{ or } (d - t - x + 1, 1^{t - i - 1}).
\end{cases}$$
\item For $0\leq i \leq t - 1$, the representation $\mathrm R_{2i+1}^{t}$ is the multiplicity-free sum of all representations $\rho_{\Delta_2,\alpha,\beta}$ where the bipartition $(\alpha,\beta)$ satisfies, for some $0 \leq x \leq d - t$,
$$\begin{cases}
\alpha = (i + x - s, s) \text{ for some } 0\leq s \leq \min(x,i),\\
\beta = (d - t + 1 - x, 1^{t - 1 - i}) \text{ or } (d - t + 2 - x, 1^{t - 2 - i}).
\end{cases}$$
\end{enumerate} 

\noindent In our notations, we used the convention that the partitions $(0)$ and $(1^0)$ are the empty partition $\emptyset$. The integer $x$ corresponds to the number of boxes we add to the first partition. We notice that if $i$ takes the maximal value, there is only one possibility for $\beta$ that is respectively $(d-t-x)$ in the first case and $(d - t + 1 - x)$ in the second case.\\
Recall from \ref{EOstratification} that the variety $X_{I}(\mathrm{id})$ is the union of the Ekedahl-Oort strata $X_{I_t}(w_t)$ for $0\leq t \leq d$ and the closure of the stratum for $t$ is the union of all strata $X_{I_s}(w_s)$ for $s \leq t$. At the level of cohomology, it translates into the following $F^2,\mathrm{U}_{2d+1}(q)$-equivariant spectral sequence 
$$\mathrm E_1^{t,i} : \mathrm H_c^{t+i}(X_{I_t}(w_t)) \implies \mathrm H_c^{t+i}(X_{I}(\mathrm{id})).$$
The first page of the sequence is drawn in the Figure 1, it has a triangular shape.

\begin{figure}[h]
\hspace{-30pt}
\begin{tikzcd}
	\, & \, & \, & \, & \, & \, & \, & \mathrm R_{2d}^{d}\\
	\, & \, & \, & \, & \, & \, & \quad \mathrm R_{2d - 2}^{d-1} \quad \arrow{r} & \mathrm R_{2d - 2}^{d} \oplus \mathrm R_{2d-1}^{d}\\
	\, & \, & \, & \, & \reflectbox{$\ddots$} & \, & \, & \vdots\\
	\, & \, & \quad \mathrm R_4^2 \quad \arrow{r} & \, & \ldots & \arrow{r} & \mathrm R_4^{d - 1} \oplus \mathrm R_5^{d-1} \arrow{r} & \mathrm R_4^{d} \oplus \mathrm R_5^{d}\\
    \, & \quad \mathrm R^1_2 \quad \arrow{r} & \mathrm R^2_2 \oplus \mathrm R^2_3 \arrow{r} & \, & \ldots & \arrow{r} & \mathrm R_2^{d - 1} \oplus \mathrm R_3^{d-1} \arrow{r} & \mathrm R_2^{d} \oplus \mathrm R_3^{d} \\
    \mathrm R_0^0 \arrow{r} & \mathrm R_0^1 \oplus \mathrm R_1^1 \arrow{r} & \mathrm R_0^2 \oplus \mathrm R_1^2 \arrow{r} & \, & \ldots & \arrow{r} & \mathrm R_0^{d-1} \oplus \mathrm R_1^{d-1} \arrow{r} & \mathrm R_0^{d}\oplus \mathrm R_1^{d}
\end{tikzcd}
\caption{The first page of the spectral sequence.}
\end{figure}

\noindent The representation $\mathrm R^{t}_{a}$ corresponds to the eigenvalue $(-q)^a$ of $F^2$ as before. The only eigenvalues of $F^2$ on the $i$-th row of the spectral sequence are $q^{2i}$ and $-q^{2i+1}$. In particular, the eigenvalues on two distinct rows are different. Since the differentials in deeper pages of the sequence map terms from different rows, their $F^2$-equivariance implies that they vanish. Therefore, the sequence degenerates on the second page.\\
Moreover, by the machinery of spectral sequences, for $0\leq k \leq 2d$ there exists a filtration by $\mathrm{U}_{2d+1}(q)\times \langle F^2 \rangle$-modules on $\mathrm H_c^{k}(X_{I}(\mathrm{id}))$ whose graded components are the terms of the second page lying on the anti-diagonal $t+i=k$. Since the group algebra $\overline{\mathbb Q_{\ell}}[\mathrm{U}_{2d+1}(q)]$ is semi-simple, the filtration splits, meaning that $\mathrm H_c^{k}(X_{I}(\mathrm{id}))$ is actually the direct sum of the graded components. The purity of $\mathrm H_c^{k}(X_{I}(\mathrm{id}))$ then implies that all the terms of the second page lying on the anti-diagonal $t+i = k$, which are associated to an eigenvalue whose modulus is not equal to $q^k$, must be zero. Therefore, the second page has the shape described in Figure 2. The Frobenius $F^2$ acts via $q^{2i}$ on the term $\mathrm E_2^{i,i}$, and via $-q^{2i+1}$ on the term $\mathrm E_2^{i+1,i}$. Point (2) of the Theorem readily follows.\\ 

\begin{figure}[h]
\begin{center}
\begin{tikzcd}
	\, & \, & \, & \, & \, & \mathrm E_2^{d,d}\\
	\, & \, & \, & \, & \mathrm E_2^{d-1,d-1} & \mathrm E_2^{d,d-1}\\
	\, & \, & \, & \reflectbox{$\ddots$} & \, & \vdots \\
    \, & \mathrm E_2^{1,1} & \mathrm E_2^{2,1} & 0 & \ldots & 0 \\
	\mathrm E_2^{0,0} & \mathrm E_2^{1,0} & 0 & 0 & \ldots & 0
\end{tikzcd}
\end{center}
\caption{The second page of the spectral sequence.}
\end{figure}

\noindent By the previous computations, we understand precisely all the terms in the first page of the spectral sequence. The key observation to compute the second page is that two terms on the first page which lie on the same row, but are separated by at least $2$ arrows, do not have any irreducible component in common. We make the argument more precise in the following two paragraphs, distinguishing the cohomology groups of even and odd index.\\

\noindent We first compute the cohomology group $\mathrm H_{c}^{2t}(X_{I}(\mathrm{id}))$ for $0\leq t \leq d$. We look at the following portion of the first page 
$$
\begin{tikzcd}
	\mathrm R^{t}_{2t} \arrow{r} & \mathrm R_{2t}^{t+1} \oplus \mathrm R_{2t +1}^{t +1} \arrow{r} & \mathrm R_{2t}^{t + 2} \oplus \mathrm R_{2t +1}^{t +2}
\end{tikzcd}.
$$
By extracting the eigenspaces corresponding to $q^{2t}$, we actually have the following sequence
$$
\begin{tikzcd}
	\mathrm R^{t}_{2t} \arrow{r}{u} & \mathrm R_{2t}^{t+1}  \arrow{r}{v} & \mathrm R_{2t}^{t + 2} \end{tikzcd}.
$$
The representation $\mathrm R^{t}_{2t}$ is the sum of all the representations $\rho_{\Delta_1,\alpha,\beta}$ where for some $0\leq x \leq d - t$ and for some $0\leq s \leq \min(x,t)$, we have $\alpha = (t + x - s, s)$ and $\beta = (d - t - x)$.\\
The representation $\mathrm R^{t+1}_{2t}$ is the sum of all the representations $\rho_{\Delta_1,\alpha',\beta'}$ where for some $0\leq x' \leq d - t - 1$ and for some $0\leq s \leq \min(x',t)$, we have $\alpha' = (t + x' - s, s)$ and $\beta' = (d - t - x')$ or $(d - t - x' - 1, 1)$.\\
The quotient space $\mathrm{Ker}(v)/\mathrm{Im}(u)$ is isomorphic to the eigenspace of $q^{2t}$ in $E_{2}^{t+1,t}$, which is zero. Besides, in the representation $\mathrm R^{t+2}_{2t}$ all the irreducible components have the shape $\rho_{\Delta_1,\alpha'',\beta''}$ with $\beta''$ a partition of length $2$ or $3$. In particular, all the representations $\rho_{\Delta_1,\alpha',\beta'}$ of $\mathrm R^{t+1}_{2t}$ with $\beta'$ a partition of length $1$ automatically lie inside $\mathrm{Ker}(v) = \mathrm{Im}(u)$. Such representations correspond to all the irreducible components $\rho_{\Delta_1,\alpha,\beta}$ of $\mathrm R^{t}_{2t}$ having $x \not = d - t$. Thus, none of them lies in $\mathrm{Ker}(u) \simeq E_2^{t,t}$.\\
The remaining components of $\mathrm R^{t}_{2t}$ are those having $x = d - t$, and they do not occur in the codomain of $u$ so that they lie in $\mathrm{Ker}(u)$. By the previous argument, they must form the whole of $\mathrm{Ker}(u)$.\\
Thus, we have proved that 
$$E_2^{t,t} \simeq \mathrm H_{c}^{2t}(X_{I}(\mathrm{id})) \simeq \mathrm{Ker}(u) = \bigoplus_{s = 0}^{\min(t,d-t)} \rho_{\Delta_1,(t-s,s),\emptyset}$$
and it coincides with the formula of point (3).\\

\noindent We now compute the cohomology group $\mathrm H_{c}^{2t+1}(X_{I}(\mathrm{id}))$ for $0\leq t \leq d-1$. We look at the following portion of the first page 
$$
\begin{tikzcd}
	\mathrm R^{t}_{2t} \arrow{r} & \mathrm R_{2t}^{t+1} \oplus \mathrm R_{2t +1}^{t +1} \arrow{r} & \mathrm R_{2t}^{t + 2} \oplus \mathrm R_{2t +1}^{t +2} \arrow{r} & \mathrm R^{t + 3}_{2t} \oplus \mathrm R^{t + 3}_{2t +1}
\end{tikzcd}.
$$
By extracting the eigenspaces corresponding to $-q^{2t+1}$, we actually have the following sequence 
$$
\begin{tikzcd}
	0 \arrow{r} & \mathrm R_{2t+1}^{t+1}  \arrow{r}{u} & \mathrm R_{2t+1}^{t + 2} \arrow{r}{v} & \mathrm R^{t + 3}_{2t +1} \end{tikzcd}.
$$
The representation $\mathrm R^{t+1}_{2t+1}$ is the sum of all the representations $\rho_{\Delta_2,\alpha,\beta}$ where for some $0\leq x \leq d - t - 1$ and for some $0\leq s \leq \min(x,t)$, we have $\alpha = (t + x - s,s)$ and $\beta = (d - t - x)$.\\
The representation $\mathrm R^{t+2}_{2t+1}$ is the sum of all the representations $\rho_{\Delta_2,\alpha',\beta'}$ where for some $0\leq x' \leq d - t - 2$ and for some $0\leq s \leq \min(x',t)$, we have $\alpha' = (t + x' -s,s)$ and $\beta' = (d - t - 1 - x', 1)$ or $(d - t - x')$.\\
The quotient space $\mathrm{Ker}(v)/\mathrm{Im}(u)$ is isomorphic to the eigenspace of $-q^{2t+1}$ in $E_{2}^{t+2,t}$, which is zero. Besides, in the representation $\mathrm R^{t+3}_{2t+1}$ all the irreducible components have the shape $\rho_{\Delta_2,\alpha'',\beta''}$ with $\beta''$ a partition of length $2$ or $3$. In particular, all the representations $\rho_{\Delta_2,\alpha',\beta'}$ of $\mathrm R^{t+2}_{2t+1}$ with $\beta'$ a partition of length $1$ automatically lie inside $\mathrm{Ker}(v) \simeq \mathrm{Im}(u)$. Such representations correspond to all the irreducible components $\rho_{\Delta_2,\alpha,\beta}$ of $\mathrm R^{t+1}_{2t+1}$ having $x \not = d - t - 1$. Thus, none of them lies in $\mathrm{Ker}(u) \simeq E_2^{t+1,t}$. \\
The remaining components of $\mathrm R^{t+1}_{2t+1}$ are those having $x = d - t - 1$, and they do not occur in the codomain of $u$ so that they lie in $\mathrm{Ker}(u)$. By the argument above, they must form the whole of $\mathrm{Ker}(u)$.\\
Thus, we have proved that 
$$E_2^{t+1,t} \simeq \mathrm H_{c}^{2t+1}(X_{I}(\mathrm{id})) \simeq \mathrm{Ker}(u) = \bigoplus_{s = 0}^{\min(d-t-1, t)} \rho_{\Delta_2,(t-1-s,s),\emptyset}$$
and one may check that it coincides with the formula of point (3).
\end{proof} 

\phantomsection
\printbibliography[heading=bibintoc, title={Bibliography}]
\markboth{Bibliography}{Bibliography}

\end{document}